\documentclass{amsart}

\usepackage{amssymb,amsmath,amsthm,latexsym,booktabs}

\theoremstyle{definition}
\newtheorem{definition}{Definition}

\theoremstyle{plain}
\newtheorem{lemma}[definition]{Lemma}
\newtheorem{proposition}[definition]{Proposition}
\newtheorem{theorem}[definition]{Theorem}

\begin{document}

\title[The hyperdeterminant of $3 \times 3 \times 2$ arrays]
{The hyperdeterminant of $3 \times 3 \times 2$ arrays, and\\
the simplest invariant of $4 \times 4 \times 2$ arrays}

\author{Murray R. Bremner}

\address{Department of Mathematics and Statistics,
University of Saskatchewan, Canada}

\email{bremner@math.usask.ca}

\begin{abstract}
We use the representation theory of Lie algebras and computational linear algebra
to obtain an explicit formula for the hyperdeterminant of a $3 \times 3 \times 2$ array:
a homogeneous polynomial of degree 12 in 18 variables with 16749 monomials and 41 distinct
integer coefficients; the monomials belong to 178 orbits under the action of
$(S_3 \times S_3 \times S_2) \rtimes S_2$.  We also obtain the simplest invariant
for a $4 \times 4 \times 2$ array:
a homogeneous polynomial of degree 8 in 32 variables with 14148 monomials and 13 distinct
integer coefficients; the monomials belong to 28 orbits under
$(S_4 \times S_4 \times S_2) \rtimes S_2$.
\end{abstract}

\maketitle


\section{Preliminaries}

The simple Lie algebra $\mathfrak{sl}_n(\mathbb{C})$ is the vector space of $n \times n$
complex matrices of trace 0
with the commutator product $[A,B] = AB - BA$.
The Cartan subalgebra is spanned by $H_i = E_{i,i} - E_{i+1,i+1}$
and the simple root vectors are $U_i = E_{i,i+1}$ where $i = 1, \dots, n{-}1$;
here $E_{ij}$ in the matrix unit with 1 in position $(i,j)$.
The Cartan matrix $\Gamma = ( \gamma_{ij} )$ determines the brackets $[H_i, U_j]$:
  \[
  [ H_i, U_j ] = \gamma_{ij} U_j,
  \qquad
  \gamma_{ij} =
  \left\{ \begin{tabular}{rl}
  $2$ & if $i=j$ \\ $-1$ & if $|i-j|=1$ \\ $0$ & if $|i-j| \ge 2$
  \end{tabular} \right.
  \]
We write $\Gamma_i$ for the $i$-th row of $\Gamma$.
The irreducible representation of $\mathfrak{sl}_n(\mathbb{C})$ on the vector space $\mathbb{C}^n$
is given by matrix-vector multiplication $A \cdot u = Au$.
We define
  \[
  \mathfrak{sl}_{pqr}(\mathbb{C}) =
  \mathfrak{sl}_p(\mathbb{C}) \oplus
  \mathfrak{sl}_q(\mathbb{C}) \oplus
  \mathfrak{sl}_r(\mathbb{C}),
  \qquad
  \mathbb{C}^{pqr} =
  \mathbb{C}^p \otimes \mathbb{C}^q \otimes \mathbb{C}^r.
  \]
The semisimple Lie algebra $\mathfrak{sl}_{pqr}(\mathbb{C})$ acts irreducibly
on the tensor product $\mathbb{C}^{pqr}$ by
  \[
  (A,B,C) \cdot (u \otimes v \otimes w) =
  (A \cdot u) \otimes v \otimes w + u \otimes (B \cdot v) \otimes w + u \otimes v \otimes (C \cdot w).
  \]
The standard basis of simple tensors $x_{ijk} = x_i \otimes x_j \otimes x_k$
allows us to identify the elements of $\mathbb{C}^{pqr}$ with $p \times q \times r$ arrays
(also called 3-dimensional matrices).
The algebra $\mathbb{C}[ x_{ijk} ]$ of polynomial functions on $\mathbb{C}^{pqr}$ is graded by
degree and generated by the subspace $\mathbb{C}^{pqr}$ of homogeneous elements of degree 1;
here we identify $\mathbb{C}^{pqr}$ with its dual $(\mathbb{C}^{pqr})^\ast$.
A basis for the homogeneous subspace of degree $d$ consists of the monomials
with exponent arrays $E$ of nonnegative integers summing to $d$:
  \begin{equation}
  \label{X}
  X = X(E) = \prod_{i,j,k} x_{ijk}^{e_{ijk}},
  \qquad
  E = ( e_{ijk} ),
  \end{equation}
The action of $\mathfrak{sl}_{pqr}(\mathbb{C})$ extends from $\mathbb{C}^{pqr}$ to $\mathbb{C}[ x_{ijk} ]$
by the derivation property,
  \[
  D \cdot ( fg ) = ( D \cdot f ) \, g + f \, ( D \cdot g ),
  \qquad
  D = (A,B,C).
  \]
Each homogeneous subspace is finite dimensional and hence completely reducible as a representation of
$\mathfrak{sl}_{pqr}(\mathbb{C})$.
The subalgebra of invariants is the direct sum of the trivial 1-dimensional representations in each degree:
the polynomials $f$ such that $D \cdot f = 0$ for all
$D \in \mathfrak{sl}_{pqr}(\mathbb{C})$.
By elementary representation theory, it suffices to assume that $H_\ell \cdot f = 0$ and $E_\ell \cdot f = 0$
for all three summands of $\mathfrak{sl}_{pqr}(\mathbb{C})$.
The results in the rest of this section are either standard or can be proved by straightforward calculation;
a simpler example of this approach appears in Bremner \cite{B}.
We mention Humphreys \cite{Humphreys} as a reference for Lie algebras and representation theory.

\begin{lemma}
The elements $H_i^{(1)}$, $H_j^{(2)}$, $H_k^{(3)}$ in the summands
$\mathfrak{sl}_p(\mathbb{C})$, $\mathfrak{sl}_q(\mathbb{C})$, $\mathfrak{sl}_r(\mathbb{C})$
act diagonally on monomials as the following linear differential operators:
  \begin{align*}
  H_i^{(1)} \cdot X
  &=
  \bigg( \frac{\partial}{\partial x_{i,j,k}} - \frac{\partial}{\partial x_{i+1,j,k}} \bigg) X
  =
  \Big( \sum_{j,k} e_{i,j,k} - e_{i+1,j,k} \Big) X
  \quad
  (1 \le i \le p{-}1),
  \\
  H_j^{(2)} \cdot X
  &=
  \bigg( \frac{\partial}{\partial x_{i,j,k}} - \frac{\partial}{\partial x_{i,j+1,k}} \bigg) X
  =
  \Big( \sum_{i,k} e_{i,j,k} - e_{i,j+1,k} \Big) X
  \quad
  (1 \le j \le q{-}1),
  \\
  H_k^{(3)} \cdot X
  &=
  \bigg( \frac{\partial}{\partial x_{i,j,k}} - \frac{\partial}{\partial x_{i,j,k+1}} \bigg) X
  =
  \Big( \sum_{i,j} e_{i,j,k} - e_{i,j,k+1} \Big) X
  \quad
  (1 \le k \le r{-}1).
  \end{align*}
\end{lemma}

\begin{definition}
The \textbf{weight} $\Omega(X)$ of a monomial is the ordered list of $p{+}q{+}r{-}3$ integers
consisting of its eigenvalues for $H_i^{(1)}$, $H_j^{(2)}$, $H_k^{(3)}$:
  \[
  \Omega(X)
  =
  ( \Omega_1, \Omega_2, \Omega_3 )
  =
  ( \omega_{11}, \dots, \omega_{1p}, \omega_{21}, \dots, \omega_{2q}, \omega_{31}, \dots, \omega_{3r} ).
  \]
The \textbf{weight space} $W(d;\Omega)$ has a basis consisting of the monomials of degree $d$ and weight $\Omega$.
The \textbf{zero weight space} in degree $d$ is the space $W(d;0)$.
\end{definition}

\begin{lemma}
A basis of $W(d;0)$ consists of the monomials $X(E)$ whose exponent arrays $E$ have entry sum $d$ and satisfy
the equal parallel slice property. That is, the following entry sums over 2-dimensional slices do not depend on
$i$, $j$, $k$
respectively:
  \[
  T_1(i) = \sum_{j,k} e_{ijk},
  \qquad
  T_2(j) =  \sum_{i,k} e_{ijk},
  \qquad
  T_3(k) =  \sum_{i,j} e_{ijk}.
  \]
Hence the degree of an invariant polynomial is a multiple of $\mathrm{LCM}(p,q,r)$.
\end{lemma}

\begin{lemma}
The elements $U_i^{(1)}$, $U_j^{(2)}$, $U_k^{(3)}$ in the summands
$\mathfrak{sl}_p(\mathbb{C})$, $\mathfrak{sl}_q(\mathbb{C})$, $\mathfrak{sl}_r(\mathbb{C})$
act on monomials as the following linear differential operators:
  \[
  U_i^{(1)}
  =
  x_{i,j,k} \frac{\partial}{\partial x_{i+1,j,k}},
  \qquad
  U_j^{(2)}
  =
  x_{i,j,k} \frac{\partial}{\partial x_{i,j+1,k}},
  \qquad
  U_k^{(3)}
  =
  x_{i,j,k} \frac{\partial}{\partial x_{i,j,k+1}}.
  \]
Hence we have the following formulas, where $E'$, $E''$, $E'''$ depend on $i, j, k$
and are the same as $E$ with the indicated exceptions:
  \begin{alignat*}{2}
  U_i^{(1)} \cdot X &= \sum_{j,k} e_{i+1,j,k} \, X(E'),
  &\qquad
  &\left\{
  \begin{array}{l}
  e'_{i+1,j,k} = e_{i+1,j,k}-1
  \\
  e'_{i,j,k} = e_{i,j,k} + 1
  \end{array}
  \right.
  \\
  U_j^{(2)} \cdot X &= \sum_{i,k} e_{i,j+1,k} \, X(E''),
  &\qquad
  &\left\{
  \begin{array}{l}
  e''_{i,j+1,k} = e_{i,j+1,k}-1
  \\
  e''_{i,j,k} = e_{i,j,k} + 1
  \end{array}
  \right.
  \\
  U_k^{(3)} \cdot X &= \sum_{i,j} e_{i,j,k+1} \, X(E'''),
  &\qquad
  &\left\{
  \begin{array}{l}
  e'''_{i,j,k+1} = e_{i,j,k+1}-1
  \\
  e'''_{i,j,k} = e_{i,j,k} + 1
  \end{array}
  \right.
  \end{alignat*}
\end{lemma}

\begin{lemma} \label{U-maps}
In each degree $d$, the actions of $U_i^{(1)}$, $U_j^{(2)}$, $U_k^{(3)}$ induce linear maps from
the zero weight space to higher weight spaces as follows:
  \begin{align*}
  &
  U_i^{(1)}\colon W(d;0) \longrightarrow W(d;\Gamma_i,0,0),
  \\
  &
  U_j^{(2)}\colon W(d;0) \longrightarrow W(d;0,\Gamma_j,0),
  \\
  &
  U_k^{(3)}\colon W(d;0) \longrightarrow W(d;0,0,\Gamma_k).
  \end{align*}
\end{lemma}

\begin{definition}
In degree $d$, the map $U(d)$ is the direct sum of all $U_i^{(1)}$, $U_j^{(2)}$, $U_k^{(3)}$:
  \[
  U(d) \colon W(d;0)
  \longrightarrow
  \bigoplus_i W(d;\Gamma_i,0,0)
  \oplus
  \bigoplus_j W(d;0,\Gamma_j,0)
  \oplus
  \bigoplus_k W(d;0,0,\Gamma_k).
  \]
\end{definition}

\begin{theorem}
In degree $d$, the invariant polynomials are the nullspace of $U(d)$.
\end{theorem}

We quote the following formula for the degree of the hyperdeterminant of a 3-dimensional array
whose dimensions belong to the ``sub-boundary'' format.

\begin{lemma} \label{degreelemma}
\emph{(Gelfand et al. \cite{GKZ}, Corollary 3.8)}
The hyperdeterminant of an array of size $p \times q \times r$ where $p \ge q \ge r$ and $p = q+r-2$ has degree
  \[
  2 \binom{q+r-2}{q-1} (q-1) (r-1).
  \]
In particular, for $p = q$ and $r = 2$ we obtain $2p(p-1)$.
\end{lemma}

We now specialize to $(p,q,r) = (3,3,2)$ until Section \ref{section442}; in this case the hyperdeterminant
has degree 12.

\begin{definition}
Let $B(d;\Omega)$ be the set of all exponent arrays $E$ for monomials $X(E)$ of degree $d$ and weight $\Omega$.
For $\Omega = 0$, we consider the \textbf{symmetry group} of $B(d;0)$, which is the semidirect product
$G = ( S_3 \times S_3 \times S_2 ) \rtimes S_2$.
The factors in the normal subgroup permute the slices in the three directions;
the remaining $S_2$ transposes the first and second $S_3$.  More precisely, if
$g = ( \alpha, \beta, \gamma, \delta ) \in G$,
then we have:
  \begin{align*}
  &
  ( \alpha \cdot E )_{i,j,k} = e_{\alpha(i),j,k},
  \qquad
  ( \beta \cdot E )_{i,j,k} = e_{i,\beta(j),k},
  \qquad
  ( \gamma \cdot E )_{i,j,k} = e_{i,j,\gamma(k)},
  \\
  &
  ( \delta \cdot E )_{ijk} = \begin{cases} e_{ijk} &\delta = () \\ e_{jik} &\delta = (12) \end{cases}
  \end{align*}
\end{definition}

\begin{definition}
The \textbf{flattening} of an exponent array $E = ( e_{ijk} )$ is the ordered list obtained by
applying lex order to the subscripts:
  \[
  \mathrm{flat}(E) = [ e_{111}, e_{112}, e_{121}, e_{122}, e_{131}, e_{132}, \dots, e_{331}, e_{332} ].
  \]
The \textbf{total order} on exponent arrays is defined to be the lex order on flattenings: that is, $E < E'$
for $\mathrm{flat}(E) = [ f_1, f_2, \dots, f_{18} ]$ and $\mathrm{flat}(E') = [ f'_1, f'_2, \dots, f'_{18} ]$
if and only if $f_i < f'_i$ where $i$ is the least index with $f_i \ne f'_i$.  The \textbf{matrix form}
of $E$ is:
  \[
  \left[ \begin{array}{ccc|ccc}
  e_{111} & e_{121} & e_{131} & e_{112} & e_{122} & e_{132} \\
  e_{211} & e_{221} & e_{231} & e_{212} & e_{222} & e_{232} \\
  e_{311} & e_{321} & e_{331} & e_{312} & e_{322} & e_{332}
  \end{array} \right]
  \]
\end{definition}


\section{Degree 6: no invariants}

We show that there are no invariant polynomials in degree 6
for $3 \times 3 \times 2$ arrays.  The next two results can be obtained easily
with a computer algebra system.

\begin{lemma}
In degree 6 there are 288 exponent arrays $E$ for monomials $X(E)$ of weight zero, forming 8 orbits under
the action of the symmetry group $G$.  The sizes of these orbits, and their minimal representatives in
matrix form, are:
  \allowdisplaybreaks
  \begin{alignat*}{4}
  &
  36
  &\quad
  \left[ \begin{array}{ccc|ccc}
  0 & 0 & 0 & 0 & 0 & 2 \\
  0 & 1 & 0 & 0 & 1 & 0 \\
  2 & 0 & 0 & 0 & 0 & 0
  \end{array} \right]
  &\qquad\qquad
  &
  72
  &\quad
  \left[ \begin{array}{ccc|ccc}
  0 & 0 & 0 & 0 & 0 & 2 \\
  0 & 1 & 0 & 1 & 0 & 0 \\
  1 & 1 & 0 & 0 & 0 & 0
  \end{array} \right]
  \\
  &
  6
  &\quad
  \left[ \begin{array}{ccc|ccc}
  0 & 0 & 1 & 0 & 0 & 1 \\
  0 & 1 & 0 & 0 & 1 & 0 \\
  1 & 0 & 0 & 1 & 0 & 0
  \end{array} \right]
  &\qquad\qquad
  &
  36
  &\quad
  \left[ \begin{array}{ccc|ccc}
  0 & 0 & 1 & 0 & 0 & 1 \\
  0 & 0 & 0 & 1 & 1 & 0 \\
  1 & 1 & 0 & 0 & 0 & 0
  \end{array} \right]
  \\
  &
  18
  &\quad
  \left[ \begin{array}{ccc|ccc}
  0 & 0 & 1 & 0 & 0 & 1 \\
  0 & 1 & 0 & 1 & 0 & 0 \\
  1 & 0 & 0 & 0 & 1 & 0
  \end{array} \right]
  &\qquad\qquad
  &
  72
  &\quad
  \left[ \begin{array}{ccc|ccc}
  0 & 0 & 0 & 0 & 1 & 1 \\
  0 & 0 & 1 & 1 & 0 & 0 \\
  1 & 1 & 0 & 0 & 0 & 0
  \end{array} \right]
  \\
  &
  36
  &\quad
  \left[ \begin{array}{ccc|ccc}
  0 & 0 & 0 & 0 & 1 & 1 \\
  1 & 0 & 0 & 0 & 0 & 1 \\
  1 & 1 & 0 & 0 & 0 & 0
  \end{array} \right]
  &\qquad\qquad
  &
  12
  &\quad
  \left[ \begin{array}{ccc|ccc}
  0 & 0 & 1 & 0 & 1 & 0 \\
  0 & 1 & 0 & 1 & 0 & 0 \\
  1 & 0 & 0 & 0 & 0 & 1
  \end{array} \right]
  \end{alignat*}
\end{lemma}

\begin{lemma} \label{degree6higher}
In degree 6 there are 204 exponent arrays $E$ for monomials $X(E)$ of each weight
$[2,-1,0,0,0]$, $[-1,2,0,0,0]$, $[0,0,2,-1,0]$, $[0,0,-1,2,0]$,
and 225 of weight $[0,0,0,0,2]$.
The minimal representatives of these five sets are:
  \allowdisplaybreaks
  \begin{alignat*}{2}
  &
  \left[ \begin{array}{ccc|ccc}
  0 & 0 & 0 & 0 & 1 & 2 \\
  0 & 1 & 0 & 0 & 0 & 0 \\
  2 & 0 & 0 & 0 & 0 & 0
  \end{array} \right]
  &\qquad\qquad
  &
  \left[ \begin{array}{ccc|ccc}
  0 & 0 & 0 & 0 & 0 & 2 \\
  0 & 2 & 0 & 1 & 0 & 0 \\
  1 & 0 & 0 & 0 & 0 & 0
  \end{array} \right]
  \\
  &
  \left[ \begin{array}{ccc|ccc}
  0 & 0 & 0 & 0 & 0 & 2 \\
  0 & 1 & 0 & 1 & 0 & 0 \\
  2 & 0 & 0 & 0 & 0 & 0
  \end{array} \right]
  &\qquad\qquad
  &
  \left[ \begin{array}{ccc|ccc}
  0 & 0 & 0 & 0 & 1 & 1 \\
  0 & 1 & 0 & 0 & 1 & 0 \\
  2 & 0 & 0 & 0 & 0 & 0
  \end{array} \right]
  \\
  &
  \left[ \begin{array}{ccc|ccc}
  0 & 0 & 0 & 0 & 0 & 2 \\
  0 & 2 & 0 & 0 & 0 & 0 \\
  2 & 0 & 0 & 0 & 0 & 0
  \end{array} \right]
  \end{alignat*}
\end{lemma}

\begin{proposition} \label{propdeg6}
There is no invariant polynomial in degree 6 for $3 \times 3 \times 2$ arrays.
\end{proposition}

\begin{proof}
We use the Maple package \texttt{LinearAlgebra} to create a zero matrix of size $513 \times 288$,
consisting of a $288 \times 288$ upper block and a $225 \times 288$ lower block.
The columns correspond to the ordered list of all weight zero monomials in degree 6.
The rows of the lower block correspond to the ordered lists of higher weight monomials.
We perform the following steps for each of the higher weights in the statement of Lemma \ref{degree6higher}
corresponding to the elements $U = U_1^{(1)}, U_2^{(1)}, U_1^{(2)}, U_2^{(2)}, U_1^{(3)}$:
  \begin{enumerate}
  \item For $j = 1, 2, \dots, 288$ compute the action of $U$ on the $j$-th weight zero monomial,
  obtain a linear combination of higher weight monomials, and store the coefficients in the
  appropriate rows of the lower block.
  \item Compute the row canonical form of the matrix using the Maple command \texttt{ReducedRowEchelonForm};
  the lower block is now zero.
  \end{enumerate}
At the end of this computation, the nullspace contains the coefficient vectors of the invariant polynomials.
During these five iterations, the rank of the matrix increases to 204, 277, 283, 286, 288;
hence the nullspace is zero.
(We emphasize that this computation was performed using rational arithmetic.)
\end{proof}


\section{Degree 12: the $3 \times 3 \times 2$ hyperdeterminant}

We show that there is a 1-dimensional space of invariant polynomials in degree 12.
We obtain an explicit formula for a basis element which is (up to sign) the hyperdeterminant of a
$3 \times 3 \times 2$ array in the sense of Gelfand et al. \cite{GKZ}.

\begin{lemma} \label{smithlemma}
Let $A$ be a matrix with entries in the ring $\mathbb{Z}$ of integers, let $r_0$ be its rank over
the field $\mathbb{Q}$ of rational numbers, and let $r_p$ be its rank over the field $\mathbb{F}_p$
with $p$ elements.  Then $r_p \le r_0$ for every prime $p$.  Hence the dimension of the nullspace
of $A$ over $\mathbb{Q}$ is no larger than the dimension of the nullspace over $\mathbb{F}_p$.
\end{lemma}

\begin{proof}
The Smith normal form $S^{(0)} = (s^0_{ij})$ over $\mathbb{Q}$ satisfies $s^0_{ii} = 1$ for
$1 \le i \le r_0$, other entries 0.  Hence the Smith normal form $T = (t_{ij})$ over $\mathbb{Z}$
satisfies $t_{ii} \in \mathbb{Z}$, $t_{ii} > 0$ for $1 \le i \le r_0$, other entries 0.
If $k \ge 0$ denotes the number of $t_{ii}$ which are divisible by $p$, then the Smith normal form
$S^{(p)} = (s^p_{ij})$ over $\mathbb{F}_p$ satisfies $s^p_{ii} = 1$ for $1 \le i \le r_0-k$,
other entries 0.  But clearly $S^{(p)}$ satisfies $s^p_{ii} = 1$ for $1 \le i \le r_p$, other entries 0.
Hence $r_p = r_0 - k$.
\end{proof}

\begin{theorem}
An explicit formula for the $3 \times 3 \times 2$ hyperdeterminant is displayed in Tables
\ref{longtable1}--\ref{longtable5}: an irreducible polynomial of degree 12 in 18 variables with
16749 monomials and 41 distinct integer coefficients; the monomials belong to 178 orbits under
the action of the symmetry group $(S_3 \times S_3 \times S_2) \rtimes S_2$.
\end{theorem}

\begin{proof}
The Maple code in Table \ref{mapleweightzero} generates the 16749 exponent arrays for
zero weight monomials in degree 12.  (\texttt{unflatten} converts an array
in list format to 3-dimensional table format.)
The Maple code in Table \ref{mapleorbits} computes the orbits of the exponent arrays for the action of the symmetry group.
(\texttt{weightzeroindex} does a binary search in the list of zero weight arrays;
\texttt{flatten} converts an array in table format to list format.)
The sizes and minimal representatives of the 178 orbits appear in Tables \ref{longtable1}--\ref{longtable5}.
The Maple code in Table \ref{maplehigherweight} computes the exponent arrays of higher weight resulting
from the action of the elements $U_1^{(1)}$, $U_2^{(1)}$; cases $U_1^{(2)}$, $U_2^{(2)}$, $U_1^{(3)}$ are
similar. The action of these elements on each monomial of weight zero is stored in \texttt{imagevectors}.
Each of the first four higher weights has 14442 monomials; the fifth has 15039 monomials.

We use the Maple package \texttt{LinearAlgebra[Modular]} to create a zero matrix of size $31788 \times 16749$,
consisting of a $16749 \times 16749$ upper block and a $15039 \times 16749$ lower block.
The columns correspond to the ordered list of all zero weight monomials in degree 12.
The rows of the lower block correspond to the ordered lists of higher weight monomials.
Following the same algorithm as in the proof of Proposition \ref{propdeg6}, we use modular arithmetic with $p = 1009$
to compute the nullspace of the matrix; see the Maple code in Table \ref{maplebigmat} for $U_1^{(1)}, U_2^{(1)}$
(other cases similar).
(\texttt{higherweightindex} does a binary search in the list of higher weight arrays.)
Modular arithmetic is essential in order to control memory usage for such a large matrix.
The five higher weights give cumulative ranks 14442, 16673, 16727, 16744, 16748;
thus the nullspace is 1-dimensional.  Using symmetric representatives modulo $p$, there are 41 distinct coefficients
in the canonical basis vector, and we verify that the coefficients are constant on orbits.
The coefficients, their multiplicities, and the corresponding orbits are displayed in Table \ref{tablecoefficients}.

Our last task is to verify this result in characteristic 0.
By Lemma \ref{smithlemma} we know that the dimension of the nullspace over $\mathbb{Q}$ is at most 1.
We create an ordered list of lists, \texttt{orbitlist}, containing the indices of weight zero monomials
in each orbit, and an ordered list, \texttt{invariant}, of the coefficients as a function of the orbit index.
We use integer arithmetic to confirm that the elements $U_1^{(1)}$, $U_2^{(1)}$, $U_1^{(2)}$, $U_2^{(2)}$, $U_1^{(3)}$
annihilate the basis vector of the nullspace.
See Table \ref{mapleinteger} for $U_1^{(1)}$, $U_2^{(1)}$; the other cases are similar.
(\texttt{compress} collects terms with equal monomials.)
In each case the final result is the empty list, and this completes the proof.
\end{proof}


\section{The simplest invariant of a $4 \times 4 \times 2$ array} \label{section442}

Similar computations give the following result.
(This invariant is not the hyperdeterminant, which has degree 24 by Lemma \ref{degreelemma}.)

\begin{theorem}
Every homogeneous invariant polynomial in the entries of a $4 \times 4 \times 2$ array
has degree a multiple of 4.  In degree 0 there is only the trivial invariant 1,
in degree 4 there are no invariants, and in degree 8 there is a 1-dimensional space of
invariants.  A basis of the invariants in degree 8 consists of the polynomial
displayed in Table \ref{table442} which has 32 variables, 14148 monomials, and 13 distinct
integer coefficients constant on 28 orbits under the action of
$(S_4 \times S_4 \times S_2) \rtimes S_2$.
\end{theorem}

In principle this method could also be used to study the invariants of degree 12 for a $4 \times 4 \times 2$ array,
but in this case there are 677232 monomials of weight zero.


\begin{table}
\begin{verbatim}
N := 12:
weightzero := {}:
for e[1] from 0 to N/3 do
for e[2] from 0 to N/3-e[1] do
for e[3] from 0 to N/3-add(e[p],p=1..2) do
for e[4] from 0 to N/3-add(e[p],p=1..3) do
for e[5] from 0 to N/3-add(e[p],p=1..4) do
e[6] := N/3-add(e[p],p=1..5):
  for e[ 7] from 0 to N/3 do
  for e[ 8] from 0 to N/3-e[7] do
  for e[ 9] from 0 to N/3-add(e[p],p=7.. 8) do
  for e[10] from 0 to N/3-add(e[p],p=7.. 9) do
  for e[11] from 0 to N/3-add(e[p],p=7..10) do
  e[12] := N/3-add(e[p],p=7..11):
    for e[13] from 0 to N/3 do
    for e[14] from 0 to N/3-e[13] do
    for e[15] from 0 to N/3-add(e[p],p=13..14) do
    for e[16] from 0 to N/3-add(e[p],p=13..15) do
    for e[17] from 0 to N/3-add(e[p],p=13..16) do
    e[18] := N/3-add(e[p],p=13..17):
    ee := unflatten( [seq(e[p],p=1..18)] ):
    isums := { seq(add(add(ee[i,j,k],k=1..2),j=1..3),i=1..3) }:
    jsums := { seq(add(add(ee[i,j,k],k=1..2),i=1..3),j=1..3) }:
    ksums := { seq(add(add(ee[i,j,k],j=1..3),i=1..3),k=1..2) }:
    if nops(isums)=1 and nops(jsums)=1 and nops(ksums)=1 then
      weightzero := weightzero union { [seq(e[p],p=1..18)] } fi
od od od od od od od od od od od od od od od:

\end{verbatim}
\caption{Maple code to generate exponent arrays of weight zero}
\label{mapleweightzero}
\begin{verbatim}
higherweight := table():
imagevectors := table():
for m to 2 do
  higherweight[1,m] := {}:
  for n to nops(weightzero) do
    x := unflatten( weightzero[n] ):
    imagevectors[1,m,n] := []:
    for j to 3 do for k to 2 do
      if x[m+1,j,k] >= 1 then
        xx := copy(x): xx[m+1,j,k] := xx[m+1,j,k] - 1:
        xx[m,j,k] := xx[m,j,k] + 1: xx := flatten( xx ):
        higherweight[1,m] := higherweight[1,m] union { xx }:
        imagevectors[1,m,n] :=
          [ op(imagevectors[1,m,n]), [ x[m+1,j,k], xx ] ]
      fi
od od od od:

\end{verbatim}
\caption{Maple code to generate exponent arrays of higher weight}
\label{maplehigherweight}
\end{table}


\begin{table}
\begin{verbatim}
groupaction := proc( g, x )
  local a, b, c, d, gx, i, j, k, x0, x1, x2, x3:
  a,b,c,d := op( g ): x0 := unflatten( x ):
  for i to 3 do for j to 3 do for k to 2 do
    x1[ a[i], j, k ] := x0[ i, j, k ] od od od:
  for i to 3 do for j to 3 do for k to 2 do
    x2[ i, b[j], k ] := x1[ i, j, k ] od od od:
  for i to 3 do for j to 3 do for k to 2 do
    x3[ i, j, c[k] ] := x2[ i, j, k ] od od od:
  if d = [1,2] then gx := x3 else
    for i to 3 do for j to 3 do for k to 2 do
      gx[ i, j, k ] := x3[ j, i, k ] od od od fi:
  RETURN( flatten(gx) )
end:
\end{verbatim}
\caption{Maple code for action of symmetry group}
\label{mapleaction}
\begin{verbatim}
S2 := combinat[permute](2): S3 := combinat[permute](3):
G := [seq(seq(seq(seq([a,b,c,d],d in S2),c in S2),b in S3),a in S3)]:
o := 0: orbitlist := []: indexlist := [seq(j,j=1..nops(weightzero))]:
while indexlist <> [] do
  o := o + 1: x := weightzero[indexlist[1]]: orbit := {}:
  for g in G do
    gx := groupaction( g, x ):
    j := weightzeroindex( gx, 1, nops(weightzero) ):
    orbit := orbit union { j }
  od:
  orbitlist := [ op(orbitlist), sort(convert(orbit,list)) ]:
  indexlist := sort(convert(convert(indexlist,set) minus orbit,list))
od:
\end{verbatim}
\caption{Maple code to compute orbits of exponent arrays}
\label{mapleorbits}
\begin{verbatim}
with( LinearAlgebra[Modular] ): PRIME := 1009:
bigcol := nops(weightzero):
bigrow := bigcol + nops(higherweight[3,1]):
bigmat := Create( PRIME, bigrow, bigcol, integer[] ):
for m to 2 do
  for j to bigcol do
    for cx in imagevectors[1,m,j] do
      c,x := op(cx):
      i := higherweightindex( x, 1, m, 1, nops(higherweight[1,m]) ):
      bigmat[bigcol+i,j] := ( bigmat[bigcol+i,j] + c ) mod PRIME
    od
  od:
  RowReduce(PRIME,bigmat,bigrow,bigcol,bigcol,0,0,'bigrank',0,0,true)
od:
MatBasis( PRIME, bigmat, bigrank, true ):

\end{verbatim}
\caption{Maple code to fill and reduce representation matrix}
\label{maplebigmat}
\end{table}


  \begin{table}
  \[
  \begin{array}{rrl}
  \text{coef} & \text{mult} & \text{orbits} \\
  -104 &   18 & 153 \\
   -68 &   36 & 150 \\
   -62 &   36 & 100 \\
   -38 &   36 & 92 \\
   -36 &   72 & 171 \\
   -34 &  144 & 86 \\
   -32 &   72 & 175 \\
   -30 &   24 & 173{,} 178 \\
   -27 &   12 & 146 \\
   -26 &  252 & 89{,} 134{,} 141{,} 147 \\
   -24 &  168 & 85{,} 93{,} 144 \\
   -22 &  288 & 108{,} 164 \\
   -20 &  432 & 71{,} 121{,} 129 \\
   -18 &  144 & 111 \\
   -16 &  369 & 28{,} 39{,} 158{,} 174 \\
   -12 &  648 & 32{,} 95{,} 97{,} 130{,} 136 \\
   -10 & 1116 & 9{,} 17{,} 58{,} 60{,} 79{,} 82{,} 101{,} 102{,} 118{,} 127{,} 137 \\
    -8 & 1098 & 13{,} 21{,} 27{,} 42{,} 45{,} 48{,} 50{,} 64{,} 73{,} 160{,} 165 \\
    -6 &  942 & 14{,} 19{,} 56{,} 65{,} 72{,} 76{,} 90{,} 94{,} 132 \\
    -4 & 1224 & 20{,} 26{,} 35{,} 46{,} 47{,} 54{,} 106{,} 154{,} 156{,} 159{,} 166 \\
    -2 & 1224 & 2{,} 3{,} 6{,} 36{,} 37{,} 53{,} 55{,} 57{,} 91{,} 155{,} 162{,} 170 \\
     1 &  216 & 1{,} 5{,} 114{,} 122 \\
     2 & 1656 & 10{,} 16{,} 22{,} 30{,} 33{,} 34{,} 41{,} 67{,} 68{,} 83{,} 87{,} 96{,} 104{,} 138 \\
     4 & 1332 & 4{,} 8{,} 12{,} 18{,} 29{,} 38{,} 40{,} 62{,} 77{,} 78{,} 84{,} 115{,} 116{,} 117{,} 123{,} 124{,} 125{,} 126{,} 143 \\
     6 &  972 & 11{,} 24{,} 31{,} 43{,} 44{,} 63{,} 99{,} 110{,} 168 \\
     8 &  540 & 7{,} 15{,} 66{,} 119{,} 128 \\
    10 &  612 & 25{,} 51{,} 52{,} 105{,} 169{,} 177 \\
    12 &  288 & 107{,} 109{,} 172 \\
    16 &  648 & 69{,} 70{,} 80{,} 120{,} 140{,} 148 \\
    18 &  792 & 61{,} 81{,} 88{,} 112{,} 133{,} 139{,} 142{,} 145 \\
    20 &  252 & 23{,} 49{,} 59 \\
    22 &  288 & 75{,} 131 \\
    24 &  144 & 161 \\
    26 &  144 & 113{,} 157 \\
    30 &   72 & 135 \\
    36 &  150 & 74{,} 98{,} 151{,} 152 \\
    38 &   72 & 167 \\
    44 &   72 & 149 \\
    50 &   72 & 103 \\
    54 &   36 & 163 \\
    76 &   36 & 176
  \end{array}
  \]
  \caption{Coefficients, multiplicities, and corresponding orbits}
  \label{tablecoefficients}
  \end{table}


  \begin{table}
  \[
  \begin{array}{rrrr|rrrr}
  \text{coef} & \text{representative}  & \text{size} & \#
  &
  \text{coef} & \text{representative}  & \text{size} & \#
  \\
     1 &
  {\tiny \left[ \begin{array}{ccc|ccc}
  0 &\!\!\! 0 &\!\!\! 0 &  0 &\!\!\! 0 &\!\!\! 4 \\
  0 &\!\!\! 2 &\!\!\! 0 &  0 &\!\!\! 2 &\!\!\! 0 \\
  4 &\!\!\! 0 &\!\!\! 0 &  0 &\!\!\! 0 &\!\!\! 0
  \end{array} \right] } &
   36 &
    1
  &
    -2 &
  {\tiny \left[ \begin{array}{ccc|ccc}
  0 &\!\!\! 0 &\!\!\! 0 &  0 &\!\!\! 0 &\!\!\! 4 \\
  0 &\!\!\! 3 &\!\!\! 0 &  0 &\!\!\! 1 &\!\!\! 0 \\
  3 &\!\!\! 0 &\!\!\! 0 &  1 &\!\!\! 0 &\!\!\! 0
  \end{array} \right] } &
   36 &
    2
  \\[8pt]
    -2 &
  {\tiny \left[ \begin{array}{ccc|ccc}
  0 &\!\!\! 0 &\!\!\! 0 &  0 &\!\!\! 0 &\!\!\! 4 \\
  0 &\!\!\! 2 &\!\!\! 0 &  1 &\!\!\! 1 &\!\!\! 0 \\
  3 &\!\!\! 1 &\!\!\! 0 &  0 &\!\!\! 0 &\!\!\! 0
  \end{array} \right] } &
  144 &
    3
  &
     4 &
  {\tiny \left[ \begin{array}{ccc|ccc}
  0 &\!\!\! 0 &\!\!\! 0 &  0 &\!\!\! 0 &\!\!\! 4 \\
  0 &\!\!\! 3 &\!\!\! 0 &  1 &\!\!\! 0 &\!\!\! 0 \\
  3 &\!\!\! 0 &\!\!\! 0 &  0 &\!\!\! 1 &\!\!\! 0
  \end{array} \right] } &
   36 &
    4
  \\[8pt]
     1 &
  {\tiny \left[ \begin{array}{ccc|ccc}
  0 &\!\!\! 0 &\!\!\! 0 &  0 &\!\!\! 0 &\!\!\! 4 \\
  0 &\!\!\! 2 &\!\!\! 0 &  2 &\!\!\! 0 &\!\!\! 0 \\
  2 &\!\!\! 2 &\!\!\! 0 &  0 &\!\!\! 0 &\!\!\! 0
  \end{array} \right] } &
   72 &
    5
  &
    -2 &
  {\tiny \left[ \begin{array}{ccc|ccc}
  0 &\!\!\! 0 &\!\!\! 0 &  0 &\!\!\! 0 &\!\!\! 4 \\
  1 &\!\!\! 1 &\!\!\! 0 &  0 &\!\!\! 2 &\!\!\! 0 \\
  3 &\!\!\! 1 &\!\!\! 0 &  0 &\!\!\! 0 &\!\!\! 0
  \end{array} \right] } &
   72 &
    6
  \\[8pt]
     8 &
  {\tiny \left[ \begin{array}{ccc|ccc}
  0 &\!\!\! 0 &\!\!\! 0 &  0 &\!\!\! 0 &\!\!\! 4 \\
  1 &\!\!\! 2 &\!\!\! 0 &  0 &\!\!\! 1 &\!\!\! 0 \\
  2 &\!\!\! 1 &\!\!\! 0 &  1 &\!\!\! 0 &\!\!\! 0
  \end{array} \right] } &
   36 &
    7
  &
     4 &
  {\tiny \left[ \begin{array}{ccc|ccc}
  0 &\!\!\! 0 &\!\!\! 0 &  0 &\!\!\! 0 &\!\!\! 4 \\
  1 &\!\!\! 1 &\!\!\! 0 &  1 &\!\!\! 1 &\!\!\! 0 \\
  2 &\!\!\! 2 &\!\!\! 0 &  0 &\!\!\! 0 &\!\!\! 0
  \end{array} \right] } &
   72 &
    8
  \\[8pt]
   -10 &
  {\tiny \left[ \begin{array}{ccc|ccc}
  0 &\!\!\! 0 &\!\!\! 0 &  0 &\!\!\! 0 &\!\!\! 4 \\
  1 &\!\!\! 2 &\!\!\! 0 &  1 &\!\!\! 0 &\!\!\! 0 \\
  2 &\!\!\! 1 &\!\!\! 0 &  0 &\!\!\! 1 &\!\!\! 0
  \end{array} \right] } &
   36 &
    9
  &
     2 &
  {\tiny \left[ \begin{array}{ccc|ccc}
  0 &\!\!\! 0 &\!\!\! 1 &  0 &\!\!\! 0 &\!\!\! 3 \\
  0 &\!\!\! 2 &\!\!\! 0 &  0 &\!\!\! 2 &\!\!\! 0 \\
  3 &\!\!\! 0 &\!\!\! 0 &  1 &\!\!\! 0 &\!\!\! 0
  \end{array} \right] } &
   36 &
   10
  \\[8pt]
     6 &
  {\tiny \left[ \begin{array}{ccc|ccc}
  0 &\!\!\! 0 &\!\!\! 1 &  0 &\!\!\! 0 &\!\!\! 3 \\
  0 &\!\!\! 1 &\!\!\! 0 &  1 &\!\!\! 2 &\!\!\! 0 \\
  3 &\!\!\! 1 &\!\!\! 0 &  0 &\!\!\! 0 &\!\!\! 0
  \end{array} \right] } &
  144 &
   11
  &
     4 &
  {\tiny \left[ \begin{array}{ccc|ccc}
  0 &\!\!\! 0 &\!\!\! 1 &  0 &\!\!\! 0 &\!\!\! 3 \\
  0 &\!\!\! 2 &\!\!\! 0 &  1 &\!\!\! 1 &\!\!\! 0 \\
  2 &\!\!\! 1 &\!\!\! 0 &  1 &\!\!\! 0 &\!\!\! 0
  \end{array} \right] } &
   72 &
   12
  \\[8pt]
    -8 &
  {\tiny \left[ \begin{array}{ccc|ccc}
  0 &\!\!\! 0 &\!\!\! 1 &  0 &\!\!\! 0 &\!\!\! 3 \\
  0 &\!\!\! 2 &\!\!\! 0 &  1 &\!\!\! 1 &\!\!\! 0 \\
  3 &\!\!\! 0 &\!\!\! 0 &  0 &\!\!\! 1 &\!\!\! 0
  \end{array} \right] } &
   72 &
   13
  &
    -6 &
  {\tiny \left[ \begin{array}{ccc|ccc}
  0 &\!\!\! 0 &\!\!\! 1 &  0 &\!\!\! 0 &\!\!\! 3 \\
  0 &\!\!\! 1 &\!\!\! 0 &  2 &\!\!\! 1 &\!\!\! 0 \\
  2 &\!\!\! 2 &\!\!\! 0 &  0 &\!\!\! 0 &\!\!\! 0
  \end{array} \right] } &
  144 &
   14
  \\[8pt]
     8 &
  {\tiny \left[ \begin{array}{ccc|ccc}
  0 &\!\!\! 0 &\!\!\! 1 &  0 &\!\!\! 0 &\!\!\! 3 \\
  0 &\!\!\! 2 &\!\!\! 0 &  2 &\!\!\! 0 &\!\!\! 0 \\
  2 &\!\!\! 1 &\!\!\! 0 &  0 &\!\!\! 1 &\!\!\! 0
  \end{array} \right] } &
   72 &
   15
  &
     2 &
  {\tiny \left[ \begin{array}{ccc|ccc}
  0 &\!\!\! 0 &\!\!\! 1 &  0 &\!\!\! 0 &\!\!\! 3 \\
  0 &\!\!\! 1 &\!\!\! 0 &  3 &\!\!\! 0 &\!\!\! 0 \\
  1 &\!\!\! 3 &\!\!\! 0 &  0 &\!\!\! 0 &\!\!\! 0
  \end{array} \right] } &
   72 &
   16
  \\[8pt]
   -10 &
  {\tiny \left[ \begin{array}{ccc|ccc}
  0 &\!\!\! 0 &\!\!\! 1 &  0 &\!\!\! 0 &\!\!\! 3 \\
  1 &\!\!\! 1 &\!\!\! 0 &  0 &\!\!\! 2 &\!\!\! 0 \\
  2 &\!\!\! 1 &\!\!\! 0 &  1 &\!\!\! 0 &\!\!\! 0
  \end{array} \right] } &
   72 &
   17
  &
     4 &
  {\tiny \left[ \begin{array}{ccc|ccc}
  0 &\!\!\! 0 &\!\!\! 1 &  0 &\!\!\! 0 &\!\!\! 3 \\
  1 &\!\!\! 1 &\!\!\! 0 &  1 &\!\!\! 1 &\!\!\! 0 \\
  1 &\!\!\! 2 &\!\!\! 0 &  1 &\!\!\! 0 &\!\!\! 0
  \end{array} \right] } &
   72 &
   18
  \\[8pt]
    -6 &
  {\tiny \left[ \begin{array}{ccc|ccc}
  0 &\!\!\! 0 &\!\!\! 2 &  0 &\!\!\! 0 &\!\!\! 2 \\
  0 &\!\!\! 2 &\!\!\! 0 &  0 &\!\!\! 2 &\!\!\! 0 \\
  2 &\!\!\! 0 &\!\!\! 0 &  2 &\!\!\! 0 &\!\!\! 0
  \end{array} \right] } &
    6 &
   19
  &
    -4 &
  {\tiny \left[ \begin{array}{ccc|ccc}
  0 &\!\!\! 0 &\!\!\! 2 &  0 &\!\!\! 0 &\!\!\! 2 \\
  0 &\!\!\! 0 &\!\!\! 0 &  1 &\!\!\! 3 &\!\!\! 0 \\
  3 &\!\!\! 1 &\!\!\! 0 &  0 &\!\!\! 0 &\!\!\! 0
  \end{array} \right] } &
   72 &
   20
  \\[8pt]
    -8 &
  {\tiny \left[ \begin{array}{ccc|ccc}
  0 &\!\!\! 0 &\!\!\! 2 &  0 &\!\!\! 0 &\!\!\! 2 \\
  0 &\!\!\! 1 &\!\!\! 0 &  1 &\!\!\! 2 &\!\!\! 0 \\
  2 &\!\!\! 1 &\!\!\! 0 &  1 &\!\!\! 0 &\!\!\! 0
  \end{array} \right] } &
   72 &
   21
  &
     2 &
  {\tiny \left[ \begin{array}{ccc|ccc}
  0 &\!\!\! 0 &\!\!\! 2 &  0 &\!\!\! 0 &\!\!\! 2 \\
  0 &\!\!\! 1 &\!\!\! 0 &  1 &\!\!\! 2 &\!\!\! 0 \\
  3 &\!\!\! 0 &\!\!\! 0 &  0 &\!\!\! 1 &\!\!\! 0
  \end{array} \right] } &
   72 &
   22
  \\[8pt]
    20 &
  {\tiny \left[ \begin{array}{ccc|ccc}
  0 &\!\!\! 0 &\!\!\! 2 &  0 &\!\!\! 0 &\!\!\! 2 \\
  0 &\!\!\! 2 &\!\!\! 0 &  1 &\!\!\! 1 &\!\!\! 0 \\
  2 &\!\!\! 0 &\!\!\! 0 &  1 &\!\!\! 1 &\!\!\! 0
  \end{array} \right] } &
   36 &
   23
  &
     6 &
  {\tiny \left[ \begin{array}{ccc|ccc}
  0 &\!\!\! 0 &\!\!\! 2 &  0 &\!\!\! 0 &\!\!\! 2 \\
  0 &\!\!\! 0 &\!\!\! 0 &  2 &\!\!\! 2 &\!\!\! 0 \\
  2 &\!\!\! 2 &\!\!\! 0 &  0 &\!\!\! 0 &\!\!\! 0
  \end{array} \right] } &
   36 &
   24
  \\[8pt]
    10 &
  {\tiny \left[ \begin{array}{ccc|ccc}
  0 &\!\!\! 0 &\!\!\! 2 &  0 &\!\!\! 0 &\!\!\! 2 \\
  0 &\!\!\! 1 &\!\!\! 0 &  2 &\!\!\! 1 &\!\!\! 0 \\
  1 &\!\!\! 2 &\!\!\! 0 &  1 &\!\!\! 0 &\!\!\! 0
  \end{array} \right] } &
   36 &
   25
  &
    -4 &
  {\tiny \left[ \begin{array}{ccc|ccc}
  0 &\!\!\! 0 &\!\!\! 2 &  0 &\!\!\! 0 &\!\!\! 2 \\
  0 &\!\!\! 1 &\!\!\! 0 &  2 &\!\!\! 1 &\!\!\! 0 \\
  2 &\!\!\! 1 &\!\!\! 0 &  0 &\!\!\! 1 &\!\!\! 0
  \end{array} \right] } &
   72 &
   26
  \\[8pt]
    -8 &
  {\tiny \left[ \begin{array}{ccc|ccc}
  0 &\!\!\! 0 &\!\!\! 2 &  0 &\!\!\! 0 &\!\!\! 2 \\
  0 &\!\!\! 2 &\!\!\! 0 &  2 &\!\!\! 0 &\!\!\! 0 \\
  2 &\!\!\! 0 &\!\!\! 0 &  0 &\!\!\! 2 &\!\!\! 0
  \end{array} \right] } &
   18 &
   27
  &
   -16 &
  {\tiny \left[ \begin{array}{ccc|ccc}
  0 &\!\!\! 0 &\!\!\! 2 &  0 &\!\!\! 0 &\!\!\! 2 \\
  1 &\!\!\! 1 &\!\!\! 0 &  1 &\!\!\! 1 &\!\!\! 0 \\
  1 &\!\!\! 1 &\!\!\! 0 &  1 &\!\!\! 1 &\!\!\! 0
  \end{array} \right] } &
    9 &
   28
  \\[8pt]
     4 &
  {\tiny \left[ \begin{array}{ccc|ccc}
  0 &\!\!\! 0 &\!\!\! 0 &  0 &\!\!\! 1 &\!\!\! 3 \\
  0 &\!\!\! 1 &\!\!\! 1 &  1 &\!\!\! 1 &\!\!\! 0 \\
  3 &\!\!\! 1 &\!\!\! 0 &  0 &\!\!\! 0 &\!\!\! 0
  \end{array} \right] } &
   72 &
   29
  &
     2 &
  {\tiny \left[ \begin{array}{ccc|ccc}
  0 &\!\!\! 0 &\!\!\! 0 &  0 &\!\!\! 1 &\!\!\! 3 \\
  0 &\!\!\! 2 &\!\!\! 0 &  1 &\!\!\! 0 &\!\!\! 1 \\
  3 &\!\!\! 1 &\!\!\! 0 &  0 &\!\!\! 0 &\!\!\! 0
  \end{array} \right] } &
  144 &
   30
  \\[8pt]
     6 &
  {\tiny \left[ \begin{array}{ccc|ccc}
  0 &\!\!\! 0 &\!\!\! 0 &  0 &\!\!\! 1 &\!\!\! 3 \\
  0 &\!\!\! 2 &\!\!\! 1 &  1 &\!\!\! 0 &\!\!\! 0 \\
  2 &\!\!\! 1 &\!\!\! 0 &  1 &\!\!\! 0 &\!\!\! 0
  \end{array} \right] } &
  144 &
   31
  &
   -12 &
  {\tiny \left[ \begin{array}{ccc|ccc}
  0 &\!\!\! 0 &\!\!\! 0 &  0 &\!\!\! 1 &\!\!\! 3 \\
  0 &\!\!\! 2 &\!\!\! 1 &  1 &\!\!\! 0 &\!\!\! 0 \\
  3 &\!\!\! 0 &\!\!\! 0 &  0 &\!\!\! 1 &\!\!\! 0
  \end{array} \right] } &
  144 &
   32
  \\[8pt]
     2 &
  {\tiny \left[ \begin{array}{ccc|ccc}
  0 &\!\!\! 0 &\!\!\! 0 &  0 &\!\!\! 1 &\!\!\! 3 \\
  0 &\!\!\! 2 &\!\!\! 0 &  1 &\!\!\! 1 &\!\!\! 0 \\
  3 &\!\!\! 0 &\!\!\! 1 &  0 &\!\!\! 0 &\!\!\! 0
  \end{array} \right] } &
  144 &
   33
  &
     2 &
  {\tiny \left[ \begin{array}{ccc|ccc}
  0 &\!\!\! 0 &\!\!\! 0 &  0 &\!\!\! 1 &\!\!\! 3 \\
  0 &\!\!\! 3 &\!\!\! 0 &  1 &\!\!\! 0 &\!\!\! 0 \\
  2 &\!\!\! 0 &\!\!\! 1 &  1 &\!\!\! 0 &\!\!\! 0
  \end{array} \right] } &
  144 &
   34
  \\[8pt]
    -4 &
  {\tiny \left[ \begin{array}{ccc|ccc}
  0 &\!\!\! 0 &\!\!\! 0 &  0 &\!\!\! 1 &\!\!\! 3 \\
  0 &\!\!\! 3 &\!\!\! 0 &  1 &\!\!\! 0 &\!\!\! 0 \\
  3 &\!\!\! 0 &\!\!\! 0 &  0 &\!\!\! 0 &\!\!\! 1
  \end{array} \right] } &
   72 &
   35
  &
    -2 &
  {\tiny \left[ \begin{array}{ccc|ccc}
  0 &\!\!\! 0 &\!\!\! 0 &  0 &\!\!\! 1 &\!\!\! 3 \\
  0 &\!\!\! 1 &\!\!\! 1 &  2 &\!\!\! 0 &\!\!\! 0 \\
  2 &\!\!\! 2 &\!\!\! 0 &  0 &\!\!\! 0 &\!\!\! 0
  \end{array} \right] } &
  144 &
   36
  \\[8pt]
    -2 &
  {\tiny \left[ \begin{array}{ccc|ccc}
  0 &\!\!\! 0 &\!\!\! 0 &  0 &\!\!\! 1 &\!\!\! 3 \\
  0 &\!\!\! 2 &\!\!\! 0 &  2 &\!\!\! 0 &\!\!\! 0 \\
  2 &\!\!\! 1 &\!\!\! 1 &  0 &\!\!\! 0 &\!\!\! 0
  \end{array} \right] } &
  144 &
   37
  &
     4 &
  {\tiny \left[ \begin{array}{ccc|ccc}
  0 &\!\!\! 0 &\!\!\! 0 &  0 &\!\!\! 1 &\!\!\! 3 \\
  1 &\!\!\! 1 &\!\!\! 0 &  0 &\!\!\! 1 &\!\!\! 1 \\
  3 &\!\!\! 1 &\!\!\! 0 &  0 &\!\!\! 0 &\!\!\! 0
  \end{array} \right] } &
   36 &
   38
  \\[8pt]
   -16 &
  {\tiny \left[ \begin{array}{ccc|ccc}
  0 &\!\!\! 0 &\!\!\! 0 &  0 &\!\!\! 1 &\!\!\! 3 \\
  1 &\!\!\! 1 &\!\!\! 1 &  0 &\!\!\! 1 &\!\!\! 0 \\
  2 &\!\!\! 1 &\!\!\! 0 &  1 &\!\!\! 0 &\!\!\! 0
  \end{array} \right] } &
  144 &
   39
  &
     4 &
  {\tiny \left[ \begin{array}{ccc|ccc}
  0 &\!\!\! 0 &\!\!\! 0 &  0 &\!\!\! 1 &\!\!\! 3 \\
  1 &\!\!\! 1 &\!\!\! 1 &  0 &\!\!\! 1 &\!\!\! 0 \\
  3 &\!\!\! 0 &\!\!\! 0 &  0 &\!\!\! 1 &\!\!\! 0
  \end{array} \right] } &
   72 &
   40
  \end{array}
  \]
  \smallskip
  \caption{The $3 \times 3 \times 2$ hyperdeterminant, orbits 1 to 40}
  \label{longtable1}
  \end{table}


  \begin{table}
  \[
  \begin{array}{rrrr|rrrr}
  \text{coef} & \text{representative}  & \text{size} & \#
  &
  \text{coef} & \text{representative}  & \text{size} & \#
  \\
     2 &
  {\tiny \left[ \begin{array}{ccc|ccc}
  0 &\!\!\! 0 &\!\!\! 0 &  0 &\!\!\! 1 &\!\!\! 3 \\
  1 &\!\!\! 1 &\!\!\! 0 &  0 &\!\!\! 2 &\!\!\! 0 \\
  3 &\!\!\! 0 &\!\!\! 1 &  0 &\!\!\! 0 &\!\!\! 0
  \end{array} \right] } &
  144 &
   41
  &
    -8 &
  {\tiny \left[ \begin{array}{ccc|ccc}
  0 &\!\!\! 0 &\!\!\! 0 &  0 &\!\!\! 1 &\!\!\! 3 \\
  1 &\!\!\! 2 &\!\!\! 0 &  0 &\!\!\! 0 &\!\!\! 1 \\
  2 &\!\!\! 1 &\!\!\! 0 &  1 &\!\!\! 0 &\!\!\! 0
  \end{array} \right] } &
   72 &
   42
  \\[8pt]
     6 &
  {\tiny \left[ \begin{array}{ccc|ccc}
  0 &\!\!\! 0 &\!\!\! 0 &  0 &\!\!\! 1 &\!\!\! 3 \\
  1 &\!\!\! 2 &\!\!\! 1 &  0 &\!\!\! 0 &\!\!\! 0 \\
  1 &\!\!\! 1 &\!\!\! 0 &  2 &\!\!\! 0 &\!\!\! 0
  \end{array} \right] } &
  144 &
   43
  &
     6 &
  {\tiny \left[ \begin{array}{ccc|ccc}
  0 &\!\!\! 0 &\!\!\! 0 &  0 &\!\!\! 1 &\!\!\! 3 \\
  1 &\!\!\! 2 &\!\!\! 1 &  0 &\!\!\! 0 &\!\!\! 0 \\
  2 &\!\!\! 0 &\!\!\! 0 &  1 &\!\!\! 1 &\!\!\! 0
  \end{array} \right] } &
  144 &
   44
  \\[8pt]
    -8 &
  {\tiny \left[ \begin{array}{ccc|ccc}
  0 &\!\!\! 0 &\!\!\! 0 &  0 &\!\!\! 1 &\!\!\! 3 \\
  1 &\!\!\! 2 &\!\!\! 0 &  0 &\!\!\! 1 &\!\!\! 0 \\
  2 &\!\!\! 0 &\!\!\! 1 &  1 &\!\!\! 0 &\!\!\! 0
  \end{array} \right] } &
  144 &
   45
  &
    -4 &
  {\tiny \left[ \begin{array}{ccc|ccc}
  0 &\!\!\! 0 &\!\!\! 0 &  0 &\!\!\! 1 &\!\!\! 3 \\
  1 &\!\!\! 0 &\!\!\! 1 &  1 &\!\!\! 1 &\!\!\! 0 \\
  2 &\!\!\! 2 &\!\!\! 0 &  0 &\!\!\! 0 &\!\!\! 0
  \end{array} \right] } &
  144 &
   46
  \\[8pt]
    -4 &
  {\tiny \left[ \begin{array}{ccc|ccc}
  0 &\!\!\! 0 &\!\!\! 0 &  0 &\!\!\! 1 &\!\!\! 3 \\
  1 &\!\!\! 1 &\!\!\! 0 &  1 &\!\!\! 0 &\!\!\! 1 \\
  2 &\!\!\! 2 &\!\!\! 0 &  0 &\!\!\! 0 &\!\!\! 0
  \end{array} \right] } &
  144 &
   47
  &
    -8 &
  {\tiny \left[ \begin{array}{ccc|ccc}
  0 &\!\!\! 0 &\!\!\! 0 &  0 &\!\!\! 1 &\!\!\! 3 \\
  1 &\!\!\! 1 &\!\!\! 1 &  1 &\!\!\! 0 &\!\!\! 0 \\
  1 &\!\!\! 2 &\!\!\! 0 &  1 &\!\!\! 0 &\!\!\! 0
  \end{array} \right] } &
  144 &
   48
  \\[8pt]
    20 &
  {\tiny \left[ \begin{array}{ccc|ccc}
  0 &\!\!\! 0 &\!\!\! 0 &  0 &\!\!\! 1 &\!\!\! 3 \\
  1 &\!\!\! 1 &\!\!\! 1 &  1 &\!\!\! 0 &\!\!\! 0 \\
  2 &\!\!\! 1 &\!\!\! 0 &  0 &\!\!\! 1 &\!\!\! 0
  \end{array} \right] } &
  144 &
   49
  &
    -8 &
  {\tiny \left[ \begin{array}{ccc|ccc}
  0 &\!\!\! 0 &\!\!\! 0 &  0 &\!\!\! 1 &\!\!\! 3 \\
  1 &\!\!\! 1 &\!\!\! 0 &  1 &\!\!\! 1 &\!\!\! 0 \\
  2 &\!\!\! 1 &\!\!\! 1 &  0 &\!\!\! 0 &\!\!\! 0
  \end{array} \right] } &
  144 &
   50
  \\[8pt]
    10 &
  {\tiny \left[ \begin{array}{ccc|ccc}
  0 &\!\!\! 0 &\!\!\! 0 &  0 &\!\!\! 1 &\!\!\! 3 \\
  1 &\!\!\! 2 &\!\!\! 0 &  1 &\!\!\! 0 &\!\!\! 0 \\
  2 &\!\!\! 0 &\!\!\! 1 &  0 &\!\!\! 1 &\!\!\! 0
  \end{array} \right] } &
  144 &
   51
  &
    10 &
  {\tiny \left[ \begin{array}{ccc|ccc}
  0 &\!\!\! 0 &\!\!\! 0 &  0 &\!\!\! 1 &\!\!\! 3 \\
  1 &\!\!\! 2 &\!\!\! 0 &  1 &\!\!\! 0 &\!\!\! 0 \\
  2 &\!\!\! 1 &\!\!\! 0 &  0 &\!\!\! 0 &\!\!\! 1
  \end{array} \right] } &
   72 &
   52
  \\[8pt]
    -2 &
  {\tiny \left[ \begin{array}{ccc|ccc}
  0 &\!\!\! 0 &\!\!\! 0 &  0 &\!\!\! 1 &\!\!\! 3 \\
  2 &\!\!\! 0 &\!\!\! 0 &  0 &\!\!\! 1 &\!\!\! 1 \\
  2 &\!\!\! 2 &\!\!\! 0 &  0 &\!\!\! 0 &\!\!\! 0
  \end{array} \right] } &
   72 &
   53
  &
    -4 &
  {\tiny \left[ \begin{array}{ccc|ccc}
  0 &\!\!\! 0 &\!\!\! 0 &  0 &\!\!\! 1 &\!\!\! 3 \\
  2 &\!\!\! 0 &\!\!\! 1 &  0 &\!\!\! 1 &\!\!\! 0 \\
  2 &\!\!\! 1 &\!\!\! 0 &  0 &\!\!\! 1 &\!\!\! 0
  \end{array} \right] } &
  144 &
   54
  \\[8pt]
    -2 &
  {\tiny \left[ \begin{array}{ccc|ccc}
  0 &\!\!\! 0 &\!\!\! 0 &  0 &\!\!\! 1 &\!\!\! 3 \\
  2 &\!\!\! 0 &\!\!\! 0 &  0 &\!\!\! 2 &\!\!\! 0 \\
  2 &\!\!\! 1 &\!\!\! 1 &  0 &\!\!\! 0 &\!\!\! 0
  \end{array} \right] } &
  144 &
   55
  &
    -6 &
  {\tiny \left[ \begin{array}{ccc|ccc}
  0 &\!\!\! 0 &\!\!\! 1 &  0 &\!\!\! 1 &\!\!\! 2 \\
  0 &\!\!\! 1 &\!\!\! 1 &  1 &\!\!\! 1 &\!\!\! 0 \\
  2 &\!\!\! 1 &\!\!\! 0 &  1 &\!\!\! 0 &\!\!\! 0
  \end{array} \right] } &
   72 &
   56
  \\[8pt]
    -2 &
  {\tiny \left[ \begin{array}{ccc|ccc}
  0 &\!\!\! 0 &\!\!\! 1 &  0 &\!\!\! 1 &\!\!\! 2 \\
  0 &\!\!\! 1 &\!\!\! 1 &  1 &\!\!\! 1 &\!\!\! 0 \\
  3 &\!\!\! 0 &\!\!\! 0 &  0 &\!\!\! 1 &\!\!\! 0
  \end{array} \right] } &
  144 &
   57
  &
   -10 &
  {\tiny \left[ \begin{array}{ccc|ccc}
  0 &\!\!\! 0 &\!\!\! 1 &  0 &\!\!\! 1 &\!\!\! 2 \\
  0 &\!\!\! 2 &\!\!\! 0 &  1 &\!\!\! 0 &\!\!\! 1 \\
  2 &\!\!\! 1 &\!\!\! 0 &  1 &\!\!\! 0 &\!\!\! 0
  \end{array} \right] } &
  144 &
   58
  \\[8pt]
    20 &
  {\tiny \left[ \begin{array}{ccc|ccc}
  0 &\!\!\! 0 &\!\!\! 1 &  0 &\!\!\! 1 &\!\!\! 2 \\
  0 &\!\!\! 2 &\!\!\! 0 &  1 &\!\!\! 0 &\!\!\! 1 \\
  3 &\!\!\! 0 &\!\!\! 0 &  0 &\!\!\! 1 &\!\!\! 0
  \end{array} \right] } &
   72 &
   59
  &
   -10 &
  {\tiny \left[ \begin{array}{ccc|ccc}
  0 &\!\!\! 0 &\!\!\! 1 &  0 &\!\!\! 1 &\!\!\! 2 \\
  0 &\!\!\! 2 &\!\!\! 1 &  1 &\!\!\! 0 &\!\!\! 0 \\
  1 &\!\!\! 1 &\!\!\! 0 &  2 &\!\!\! 0 &\!\!\! 0
  \end{array} \right] } &
   72 &
   60
  \\[8pt]
    18 &
  {\tiny \left[ \begin{array}{ccc|ccc}
  0 &\!\!\! 0 &\!\!\! 1 &  0 &\!\!\! 1 &\!\!\! 2 \\
  0 &\!\!\! 2 &\!\!\! 1 &  1 &\!\!\! 0 &\!\!\! 0 \\
  2 &\!\!\! 0 &\!\!\! 0 &  1 &\!\!\! 1 &\!\!\! 0
  \end{array} \right] } &
  144 &
   61
  &
     4 &
  {\tiny \left[ \begin{array}{ccc|ccc}
  0 &\!\!\! 0 &\!\!\! 1 &  0 &\!\!\! 1 &\!\!\! 2 \\
  0 &\!\!\! 2 &\!\!\! 0 &  1 &\!\!\! 1 &\!\!\! 0 \\
  2 &\!\!\! 0 &\!\!\! 1 &  1 &\!\!\! 0 &\!\!\! 0
  \end{array} \right] } &
  144 &
   62
  \\[8pt]
     6 &
  {\tiny \left[ \begin{array}{ccc|ccc}
  0 &\!\!\! 0 &\!\!\! 1 &  0 &\!\!\! 1 &\!\!\! 2 \\
  0 &\!\!\! 2 &\!\!\! 0 &  1 &\!\!\! 1 &\!\!\! 0 \\
  3 &\!\!\! 0 &\!\!\! 0 &  0 &\!\!\! 0 &\!\!\! 1
  \end{array} \right] } &
  144 &
   63
  &
    -8 &
  {\tiny \left[ \begin{array}{ccc|ccc}
  0 &\!\!\! 0 &\!\!\! 1 &  0 &\!\!\! 1 &\!\!\! 2 \\
  0 &\!\!\! 3 &\!\!\! 0 &  1 &\!\!\! 0 &\!\!\! 0 \\
  1 &\!\!\! 0 &\!\!\! 1 &  2 &\!\!\! 0 &\!\!\! 0
  \end{array} \right] } &
   72 &
   64
  \\[8pt]
    -6 &
  {\tiny \left[ \begin{array}{ccc|ccc}
  0 &\!\!\! 0 &\!\!\! 1 &  0 &\!\!\! 1 &\!\!\! 2 \\
  0 &\!\!\! 0 &\!\!\! 1 &  2 &\!\!\! 1 &\!\!\! 0 \\
  2 &\!\!\! 2 &\!\!\! 0 &  0 &\!\!\! 0 &\!\!\! 0
  \end{array} \right] } &
  144 &
   65
  &
     8 &
  {\tiny \left[ \begin{array}{ccc|ccc}
  0 &\!\!\! 0 &\!\!\! 1 &  0 &\!\!\! 1 &\!\!\! 2 \\
  0 &\!\!\! 1 &\!\!\! 0 &  2 &\!\!\! 0 &\!\!\! 1 \\
  2 &\!\!\! 2 &\!\!\! 0 &  0 &\!\!\! 0 &\!\!\! 0
  \end{array} \right] } &
  144 &
   66
  \\[8pt]
     2 &
  {\tiny \left[ \begin{array}{ccc|ccc}
  0 &\!\!\! 0 &\!\!\! 1 &  0 &\!\!\! 1 &\!\!\! 2 \\
  0 &\!\!\! 1 &\!\!\! 1 &  2 &\!\!\! 0 &\!\!\! 0 \\
  1 &\!\!\! 2 &\!\!\! 0 &  1 &\!\!\! 0 &\!\!\! 0
  \end{array} \right] } &
  144 &
   67
  &
     2 &
  {\tiny \left[ \begin{array}{ccc|ccc}
  0 &\!\!\! 0 &\!\!\! 1 &  0 &\!\!\! 1 &\!\!\! 2 \\
  0 &\!\!\! 1 &\!\!\! 1 &  2 &\!\!\! 0 &\!\!\! 0 \\
  2 &\!\!\! 1 &\!\!\! 0 &  0 &\!\!\! 1 &\!\!\! 0
  \end{array} \right] } &
  144 &
   68
  \\[8pt]
    16 &
  {\tiny \left[ \begin{array}{ccc|ccc}
  0 &\!\!\! 0 &\!\!\! 1 &  0 &\!\!\! 1 &\!\!\! 2 \\
  0 &\!\!\! 1 &\!\!\! 0 &  2 &\!\!\! 1 &\!\!\! 0 \\
  2 &\!\!\! 1 &\!\!\! 1 &  0 &\!\!\! 0 &\!\!\! 0
  \end{array} \right] } &
  144 &
   69
  &
    16 &
  {\tiny \left[ \begin{array}{ccc|ccc}
  0 &\!\!\! 0 &\!\!\! 1 &  0 &\!\!\! 1 &\!\!\! 2 \\
  0 &\!\!\! 2 &\!\!\! 0 &  2 &\!\!\! 0 &\!\!\! 0 \\
  1 &\!\!\! 1 &\!\!\! 1 &  1 &\!\!\! 0 &\!\!\! 0
  \end{array} \right] } &
  144 &
   70
  \\[8pt]
   -20 &
  {\tiny \left[ \begin{array}{ccc|ccc}
  0 &\!\!\! 0 &\!\!\! 1 &  0 &\!\!\! 1 &\!\!\! 2 \\
  0 &\!\!\! 2 &\!\!\! 0 &  2 &\!\!\! 0 &\!\!\! 0 \\
  2 &\!\!\! 0 &\!\!\! 1 &  0 &\!\!\! 1 &\!\!\! 0
  \end{array} \right] } &
  144 &
   71
  &
    -6 &
  {\tiny \left[ \begin{array}{ccc|ccc}
  0 &\!\!\! 0 &\!\!\! 1 &  0 &\!\!\! 1 &\!\!\! 2 \\
  0 &\!\!\! 2 &\!\!\! 0 &  2 &\!\!\! 0 &\!\!\! 0 \\
  2 &\!\!\! 1 &\!\!\! 0 &  0 &\!\!\! 0 &\!\!\! 1
  \end{array} \right] } &
   72 &
   72
  \\[8pt]
    -8 &
  {\tiny \left[ \begin{array}{ccc|ccc}
  0 &\!\!\! 0 &\!\!\! 1 &  0 &\!\!\! 1 &\!\!\! 2 \\
  0 &\!\!\! 1 &\!\!\! 0 &  3 &\!\!\! 0 &\!\!\! 0 \\
  1 &\!\!\! 2 &\!\!\! 1 &  0 &\!\!\! 0 &\!\!\! 0
  \end{array} \right] } &
  144 &
   73
  &
    36 &
  {\tiny \left[ \begin{array}{ccc|ccc}
  0 &\!\!\! 0 &\!\!\! 1 &  0 &\!\!\! 1 &\!\!\! 2 \\
  1 &\!\!\! 1 &\!\!\! 0 &  0 &\!\!\! 1 &\!\!\! 1 \\
  2 &\!\!\! 1 &\!\!\! 0 &  1 &\!\!\! 0 &\!\!\! 0
  \end{array} \right] } &
   36 &
   74
  \\[8pt]
    22 &
  {\tiny \left[ \begin{array}{ccc|ccc}
  0 &\!\!\! 0 &\!\!\! 1 &  0 &\!\!\! 1 &\!\!\! 2 \\
  1 &\!\!\! 1 &\!\!\! 1 &  0 &\!\!\! 1 &\!\!\! 0 \\
  1 &\!\!\! 1 &\!\!\! 0 &  2 &\!\!\! 0 &\!\!\! 0
  \end{array} \right] } &
  144 &
   75
  &
    -6 &
  {\tiny \left[ \begin{array}{ccc|ccc}
  0 &\!\!\! 0 &\!\!\! 1 &  0 &\!\!\! 1 &\!\!\! 2 \\
  1 &\!\!\! 1 &\!\!\! 1 &  0 &\!\!\! 1 &\!\!\! 0 \\
  2 &\!\!\! 0 &\!\!\! 0 &  1 &\!\!\! 1 &\!\!\! 0
  \end{array} \right] } &
   72 &
   76
  \\[8pt]
     4 &
  {\tiny \left[ \begin{array}{ccc|ccc}
  0 &\!\!\! 0 &\!\!\! 1 &  0 &\!\!\! 1 &\!\!\! 2 \\
  1 &\!\!\! 1 &\!\!\! 0 &  0 &\!\!\! 2 &\!\!\! 0 \\
  2 &\!\!\! 0 &\!\!\! 1 &  1 &\!\!\! 0 &\!\!\! 0
  \end{array} \right] } &
   72 &
   77
  &
     4 &
  {\tiny \left[ \begin{array}{ccc|ccc}
  0 &\!\!\! 0 &\!\!\! 1 &  0 &\!\!\! 1 &\!\!\! 2 \\
  1 &\!\!\! 2 &\!\!\! 0 &  0 &\!\!\! 0 &\!\!\! 1 \\
  1 &\!\!\! 1 &\!\!\! 0 &  2 &\!\!\! 0 &\!\!\! 0
  \end{array} \right] } &
   72 &
   78
  \\[8pt]
   -10 &
  {\tiny \left[ \begin{array}{ccc|ccc}
  0 &\!\!\! 0 &\!\!\! 1 &  0 &\!\!\! 1 &\!\!\! 2 \\
  1 &\!\!\! 2 &\!\!\! 1 &  0 &\!\!\! 0 &\!\!\! 0 \\
  1 &\!\!\! 0 &\!\!\! 0 &  2 &\!\!\! 1 &\!\!\! 0
  \end{array} \right] } &
   72 &
   79
  &
    16 &
  {\tiny \left[ \begin{array}{ccc|ccc}
  0 &\!\!\! 0 &\!\!\! 1 &  0 &\!\!\! 1 &\!\!\! 2 \\
  1 &\!\!\! 0 &\!\!\! 0 &  1 &\!\!\! 1 &\!\!\! 1 \\
  2 &\!\!\! 2 &\!\!\! 0 &  0 &\!\!\! 0 &\!\!\! 0
  \end{array} \right] } &
  144 &
   80
  \end{array}
  \]
  \smallskip
  \caption{The $3 \times 3 \times 2$ hyperdeterminant, orbits 41 to 80}
  \label{longtable2}
  \end{table}


  \begin{table}
  \[
  \begin{array}{rrrr|rrrr}
  \text{coef} & \text{representative}  & \text{size} & \#
  &
  \text{coef} & \text{representative}  & \text{size} & \#
  \\
    18 &
  {\tiny \left[ \begin{array}{ccc|ccc}
  0 &\!\!\! 0 &\!\!\! 1 &  0 &\!\!\! 1 &\!\!\! 2 \\
  1 &\!\!\! 0 &\!\!\! 1 &  1 &\!\!\! 1 &\!\!\! 0 \\
  1 &\!\!\! 2 &\!\!\! 0 &  1 &\!\!\! 0 &\!\!\! 0
  \end{array} \right] } &
  144 &
   81
  &
   -10 &
  {\tiny \left[ \begin{array}{ccc|ccc}
  0 &\!\!\! 0 &\!\!\! 1 &  0 &\!\!\! 1 &\!\!\! 2 \\
  1 &\!\!\! 0 &\!\!\! 1 &  1 &\!\!\! 1 &\!\!\! 0 \\
  2 &\!\!\! 1 &\!\!\! 0 &  0 &\!\!\! 1 &\!\!\! 0
  \end{array} \right] } &
  144 &
   82
  \\[8pt]
     2 &
  {\tiny \left[ \begin{array}{ccc|ccc}
  0 &\!\!\! 0 &\!\!\! 1 &  0 &\!\!\! 1 &\!\!\! 2 \\
  1 &\!\!\! 0 &\!\!\! 0 &  1 &\!\!\! 2 &\!\!\! 0 \\
  2 &\!\!\! 1 &\!\!\! 1 &  0 &\!\!\! 0 &\!\!\! 0
  \end{array} \right] } &
  144 &
   83
  &
     4 &
  {\tiny \left[ \begin{array}{ccc|ccc}
  0 &\!\!\! 0 &\!\!\! 1 &  0 &\!\!\! 1 &\!\!\! 2 \\
  1 &\!\!\! 1 &\!\!\! 0 &  1 &\!\!\! 0 &\!\!\! 1 \\
  1 &\!\!\! 2 &\!\!\! 0 &  1 &\!\!\! 0 &\!\!\! 0
  \end{array} \right] } &
  144 &
   84
  \\[8pt]
   -24 &
  {\tiny \left[ \begin{array}{ccc|ccc}
  0 &\!\!\! 0 &\!\!\! 1 &  0 &\!\!\! 1 &\!\!\! 2 \\
  1 &\!\!\! 1 &\!\!\! 0 &  1 &\!\!\! 0 &\!\!\! 1 \\
  2 &\!\!\! 1 &\!\!\! 0 &  0 &\!\!\! 1 &\!\!\! 0
  \end{array} \right] } &
   72 &
   85
  &
   -34 &
  {\tiny \left[ \begin{array}{ccc|ccc}
  0 &\!\!\! 0 &\!\!\! 1 &  0 &\!\!\! 1 &\!\!\! 2 \\
  1 &\!\!\! 1 &\!\!\! 1 &  1 &\!\!\! 0 &\!\!\! 0 \\
  1 &\!\!\! 1 &\!\!\! 0 &  1 &\!\!\! 1 &\!\!\! 0
  \end{array} \right] } &
  144 &
   86
  \\[8pt]
     2 &
  {\tiny \left[ \begin{array}{ccc|ccc}
  0 &\!\!\! 0 &\!\!\! 1 &  0 &\!\!\! 1 &\!\!\! 2 \\
  1 &\!\!\! 1 &\!\!\! 1 &  1 &\!\!\! 0 &\!\!\! 0 \\
  2 &\!\!\! 0 &\!\!\! 0 &  0 &\!\!\! 2 &\!\!\! 0
  \end{array} \right] } &
  144 &
   87
  &
    18 &
  {\tiny \left[ \begin{array}{ccc|ccc}
  0 &\!\!\! 0 &\!\!\! 1 &  0 &\!\!\! 1 &\!\!\! 2 \\
  1 &\!\!\! 1 &\!\!\! 0 &  1 &\!\!\! 1 &\!\!\! 0 \\
  2 &\!\!\! 0 &\!\!\! 1 &  0 &\!\!\! 1 &\!\!\! 0
  \end{array} \right] } &
  144 &
   88
  \\[8pt]
   -26 &
  {\tiny \left[ \begin{array}{ccc|ccc}
  0 &\!\!\! 0 &\!\!\! 1 &  0 &\!\!\! 1 &\!\!\! 2 \\
  1 &\!\!\! 2 &\!\!\! 0 &  1 &\!\!\! 0 &\!\!\! 0 \\
  2 &\!\!\! 0 &\!\!\! 0 &  0 &\!\!\! 1 &\!\!\! 1
  \end{array} \right] } &
   72 &
   89
  &
    -6 &
  {\tiny \left[ \begin{array}{ccc|ccc}
  0 &\!\!\! 0 &\!\!\! 1 &  0 &\!\!\! 1 &\!\!\! 2 \\
  2 &\!\!\! 0 &\!\!\! 1 &  0 &\!\!\! 1 &\!\!\! 0 \\
  2 &\!\!\! 0 &\!\!\! 0 &  0 &\!\!\! 2 &\!\!\! 0
  \end{array} \right] } &
  144 &
   90
  \\[8pt]
    -2 &
  {\tiny \left[ \begin{array}{ccc|ccc}
  0 &\!\!\! 0 &\!\!\! 2 &  0 &\!\!\! 1 &\!\!\! 1 \\
  0 &\!\!\! 1 &\!\!\! 0 &  1 &\!\!\! 1 &\!\!\! 1 \\
  3 &\!\!\! 0 &\!\!\! 0 &  0 &\!\!\! 1 &\!\!\! 0
  \end{array} \right] } &
   72 &
   91
  &
   -38 &
  {\tiny \left[ \begin{array}{ccc|ccc}
  0 &\!\!\! 0 &\!\!\! 2 &  0 &\!\!\! 1 &\!\!\! 1 \\
  0 &\!\!\! 2 &\!\!\! 0 &  1 &\!\!\! 0 &\!\!\! 1 \\
  2 &\!\!\! 0 &\!\!\! 0 &  1 &\!\!\! 1 &\!\!\! 0
  \end{array} \right] } &
   36 &
   92
  \\[8pt]
   -24 &
  {\tiny \left[ \begin{array}{ccc|ccc}
  0 &\!\!\! 0 &\!\!\! 2 &  0 &\!\!\! 1 &\!\!\! 1 \\
  0 &\!\!\! 2 &\!\!\! 0 &  1 &\!\!\! 1 &\!\!\! 0 \\
  2 &\!\!\! 0 &\!\!\! 0 &  1 &\!\!\! 0 &\!\!\! 1
  \end{array} \right] } &
   24 &
   93
  &
    -6 &
  {\tiny \left[ \begin{array}{ccc|ccc}
  0 &\!\!\! 0 &\!\!\! 2 &  0 &\!\!\! 1 &\!\!\! 1 \\
  0 &\!\!\! 0 &\!\!\! 0 &  2 &\!\!\! 1 &\!\!\! 1 \\
  2 &\!\!\! 2 &\!\!\! 0 &  0 &\!\!\! 0 &\!\!\! 0
  \end{array} \right] } &
  144 &
   94
  \\[8pt]
   -12 &
  {\tiny \left[ \begin{array}{ccc|ccc}
  0 &\!\!\! 0 &\!\!\! 2 &  0 &\!\!\! 1 &\!\!\! 1 \\
  0 &\!\!\! 1 &\!\!\! 0 &  2 &\!\!\! 0 &\!\!\! 1 \\
  1 &\!\!\! 2 &\!\!\! 0 &  1 &\!\!\! 0 &\!\!\! 0
  \end{array} \right] } &
  144 &
   95
  &
     2 &
  {\tiny \left[ \begin{array}{ccc|ccc}
  0 &\!\!\! 0 &\!\!\! 2 &  0 &\!\!\! 1 &\!\!\! 1 \\
  0 &\!\!\! 1 &\!\!\! 0 &  2 &\!\!\! 0 &\!\!\! 1 \\
  2 &\!\!\! 1 &\!\!\! 0 &  0 &\!\!\! 1 &\!\!\! 0
  \end{array} \right] } &
  144 &
   96
  \\[8pt]
   -12 &
  {\tiny \left[ \begin{array}{ccc|ccc}
  0 &\!\!\! 0 &\!\!\! 2 &  0 &\!\!\! 1 &\!\!\! 1 \\
  0 &\!\!\! 1 &\!\!\! 0 &  2 &\!\!\! 1 &\!\!\! 0 \\
  2 &\!\!\! 1 &\!\!\! 0 &  0 &\!\!\! 0 &\!\!\! 1
  \end{array} \right] } &
  144 &
   97
  &
    36 &
  {\tiny \left[ \begin{array}{ccc|ccc}
  0 &\!\!\! 0 &\!\!\! 2 &  0 &\!\!\! 1 &\!\!\! 1 \\
  0 &\!\!\! 2 &\!\!\! 0 &  2 &\!\!\! 0 &\!\!\! 0 \\
  2 &\!\!\! 0 &\!\!\! 0 &  0 &\!\!\! 1 &\!\!\! 1
  \end{array} \right] } &
   72 &
   98
  \\[8pt]
     6 &
  {\tiny \left[ \begin{array}{ccc|ccc}
  0 &\!\!\! 0 &\!\!\! 2 &  0 &\!\!\! 1 &\!\!\! 1 \\
  0 &\!\!\! 1 &\!\!\! 0 &  3 &\!\!\! 0 &\!\!\! 0 \\
  1 &\!\!\! 2 &\!\!\! 0 &  0 &\!\!\! 0 &\!\!\! 1
  \end{array} \right] } &
   72 &
   99
  &
   -62 &
  {\tiny \left[ \begin{array}{ccc|ccc}
  0 &\!\!\! 0 &\!\!\! 2 &  0 &\!\!\! 1 &\!\!\! 1 \\
  1 &\!\!\! 1 &\!\!\! 0 &  0 &\!\!\! 1 &\!\!\! 1 \\
  1 &\!\!\! 1 &\!\!\! 0 &  2 &\!\!\! 0 &\!\!\! 0
  \end{array} \right] } &
   36 &
  100
  \\[8pt]
   -10 &
  {\tiny \left[ \begin{array}{ccc|ccc}
  0 &\!\!\! 0 &\!\!\! 2 &  0 &\!\!\! 1 &\!\!\! 1 \\
  1 &\!\!\! 0 &\!\!\! 0 &  1 &\!\!\! 1 &\!\!\! 1 \\
  1 &\!\!\! 2 &\!\!\! 0 &  1 &\!\!\! 0 &\!\!\! 0
  \end{array} \right] } &
  144 &
  101
  &
   -10 &
  {\tiny \left[ \begin{array}{ccc|ccc}
  0 &\!\!\! 0 &\!\!\! 2 &  0 &\!\!\! 1 &\!\!\! 1 \\
  1 &\!\!\! 0 &\!\!\! 0 &  1 &\!\!\! 1 &\!\!\! 1 \\
  2 &\!\!\! 1 &\!\!\! 0 &  0 &\!\!\! 1 &\!\!\! 0
  \end{array} \right] } &
   72 &
  102
  \\[8pt]
    50 &
  {\tiny \left[ \begin{array}{ccc|ccc}
  0 &\!\!\! 0 &\!\!\! 2 &  0 &\!\!\! 1 &\!\!\! 1 \\
  1 &\!\!\! 1 &\!\!\! 0 &  1 &\!\!\! 0 &\!\!\! 1 \\
  1 &\!\!\! 1 &\!\!\! 0 &  1 &\!\!\! 1 &\!\!\! 0
  \end{array} \right] } &
   72 &
  103
  &
     2 &
  {\tiny \left[ \begin{array}{ccc|ccc}
  0 &\!\!\! 0 &\!\!\! 3 &  0 &\!\!\! 1 &\!\!\! 0 \\
  0 &\!\!\! 0 &\!\!\! 0 &  1 &\!\!\! 2 &\!\!\! 1 \\
  3 &\!\!\! 0 &\!\!\! 0 &  0 &\!\!\! 1 &\!\!\! 0
  \end{array} \right] } &
   36 &
  104
  \\[8pt]
    10 &
  {\tiny \left[ \begin{array}{ccc|ccc}
  0 &\!\!\! 0 &\!\!\! 3 &  0 &\!\!\! 1 &\!\!\! 0 \\
  0 &\!\!\! 0 &\!\!\! 0 &  2 &\!\!\! 1 &\!\!\! 1 \\
  1 &\!\!\! 2 &\!\!\! 0 &  1 &\!\!\! 0 &\!\!\! 0
  \end{array} \right] } &
  144 &
  105
  &
    -4 &
  {\tiny \left[ \begin{array}{ccc|ccc}
  0 &\!\!\! 0 &\!\!\! 3 &  0 &\!\!\! 1 &\!\!\! 0 \\
  0 &\!\!\! 0 &\!\!\! 0 &  2 &\!\!\! 1 &\!\!\! 1 \\
  2 &\!\!\! 1 &\!\!\! 0 &  0 &\!\!\! 1 &\!\!\! 0
  \end{array} \right] } &
  144 &
  106
  \\[8pt]
    12 &
  {\tiny \left[ \begin{array}{ccc|ccc}
  0 &\!\!\! 0 &\!\!\! 3 &  0 &\!\!\! 1 &\!\!\! 0 \\
  0 &\!\!\! 0 &\!\!\! 0 &  2 &\!\!\! 2 &\!\!\! 0 \\
  2 &\!\!\! 1 &\!\!\! 0 &  0 &\!\!\! 0 &\!\!\! 1
  \end{array} \right] } &
  144 &
  107
  &
   -22 &
  {\tiny \left[ \begin{array}{ccc|ccc}
  0 &\!\!\! 0 &\!\!\! 3 &  0 &\!\!\! 1 &\!\!\! 0 \\
  0 &\!\!\! 1 &\!\!\! 0 &  2 &\!\!\! 0 &\!\!\! 1 \\
  1 &\!\!\! 1 &\!\!\! 0 &  1 &\!\!\! 1 &\!\!\! 0
  \end{array} \right] } &
  144 &
  108
  \\[8pt]
    12 &
  {\tiny \left[ \begin{array}{ccc|ccc}
  0 &\!\!\! 0 &\!\!\! 3 &  0 &\!\!\! 1 &\!\!\! 0 \\
  0 &\!\!\! 1 &\!\!\! 0 &  2 &\!\!\! 0 &\!\!\! 1 \\
  2 &\!\!\! 0 &\!\!\! 0 &  0 &\!\!\! 2 &\!\!\! 0
  \end{array} \right] } &
   72 &
  109
  &
     6 &
  {\tiny \left[ \begin{array}{ccc|ccc}
  0 &\!\!\! 0 &\!\!\! 3 &  0 &\!\!\! 1 &\!\!\! 0 \\
  0 &\!\!\! 1 &\!\!\! 0 &  2 &\!\!\! 1 &\!\!\! 0 \\
  1 &\!\!\! 1 &\!\!\! 0 &  1 &\!\!\! 0 &\!\!\! 1
  \end{array} \right] } &
   72 &
  110
  \\[8pt]
   -18 &
  {\tiny \left[ \begin{array}{ccc|ccc}
  0 &\!\!\! 0 &\!\!\! 3 &  0 &\!\!\! 1 &\!\!\! 0 \\
  0 &\!\!\! 1 &\!\!\! 0 &  2 &\!\!\! 1 &\!\!\! 0 \\
  2 &\!\!\! 0 &\!\!\! 0 &  0 &\!\!\! 1 &\!\!\! 1
  \end{array} \right] } &
  144 &
  111
  &
    18 &
  {\tiny \left[ \begin{array}{ccc|ccc}
  0 &\!\!\! 0 &\!\!\! 3 &  0 &\!\!\! 1 &\!\!\! 0 \\
  0 &\!\!\! 1 &\!\!\! 0 &  3 &\!\!\! 0 &\!\!\! 0 \\
  1 &\!\!\! 1 &\!\!\! 0 &  0 &\!\!\! 1 &\!\!\! 1
  \end{array} \right] } &
   36 &
  112
  \\[8pt]
    26 &
  {\tiny \left[ \begin{array}{ccc|ccc}
  0 &\!\!\! 0 &\!\!\! 3 &  0 &\!\!\! 1 &\!\!\! 0 \\
  1 &\!\!\! 0 &\!\!\! 0 &  1 &\!\!\! 1 &\!\!\! 1 \\
  1 &\!\!\! 1 &\!\!\! 0 &  1 &\!\!\! 1 &\!\!\! 0
  \end{array} \right] } &
   72 &
  113
  &
     1 &
  {\tiny \left[ \begin{array}{ccc|ccc}
  0 &\!\!\! 0 &\!\!\! 0 &  0 &\!\!\! 2 &\!\!\! 2 \\
  0 &\!\!\! 0 &\!\!\! 2 &  2 &\!\!\! 0 &\!\!\! 0 \\
  2 &\!\!\! 2 &\!\!\! 0 &  0 &\!\!\! 0 &\!\!\! 0
  \end{array} \right] } &
   72 &
  114
  \\[8pt]
     4 &
  {\tiny \left[ \begin{array}{ccc|ccc}
  0 &\!\!\! 0 &\!\!\! 0 &  0 &\!\!\! 2 &\!\!\! 2 \\
  0 &\!\!\! 1 &\!\!\! 1 &  2 &\!\!\! 0 &\!\!\! 0 \\
  2 &\!\!\! 1 &\!\!\! 1 &  0 &\!\!\! 0 &\!\!\! 0
  \end{array} \right] } &
   72 &
  115
  &
     4 &
  {\tiny \left[ \begin{array}{ccc|ccc}
  0 &\!\!\! 0 &\!\!\! 0 &  0 &\!\!\! 2 &\!\!\! 2 \\
  1 &\!\!\! 0 &\!\!\! 1 &  1 &\!\!\! 0 &\!\!\! 1 \\
  2 &\!\!\! 2 &\!\!\! 0 &  0 &\!\!\! 0 &\!\!\! 0
  \end{array} \right] } &
   72 &
  116
  \\[8pt]
     4 &
  {\tiny \left[ \begin{array}{ccc|ccc}
  0 &\!\!\! 0 &\!\!\! 0 &  0 &\!\!\! 2 &\!\!\! 2 \\
  1 &\!\!\! 0 &\!\!\! 2 &  1 &\!\!\! 0 &\!\!\! 0 \\
  1 &\!\!\! 2 &\!\!\! 0 &  1 &\!\!\! 0 &\!\!\! 0
  \end{array} \right] } &
   72 &
  117
  &
   -10 &
  {\tiny \left[ \begin{array}{ccc|ccc}
  0 &\!\!\! 0 &\!\!\! 0 &  0 &\!\!\! 2 &\!\!\! 2 \\
  1 &\!\!\! 0 &\!\!\! 2 &  1 &\!\!\! 0 &\!\!\! 0 \\
  2 &\!\!\! 1 &\!\!\! 0 &  0 &\!\!\! 1 &\!\!\! 0
  \end{array} \right] } &
  144 &
  118
  \\[8pt]
     8 &
  {\tiny \left[ \begin{array}{ccc|ccc}
  0 &\!\!\! 0 &\!\!\! 0 &  0 &\!\!\! 2 &\!\!\! 2 \\
  1 &\!\!\! 0 &\!\!\! 1 &  1 &\!\!\! 1 &\!\!\! 0 \\
  2 &\!\!\! 1 &\!\!\! 1 &  0 &\!\!\! 0 &\!\!\! 0
  \end{array} \right] } &
  144 &
  119
  &
    16 &
  {\tiny \left[ \begin{array}{ccc|ccc}
  0 &\!\!\! 0 &\!\!\! 0 &  0 &\!\!\! 2 &\!\!\! 2 \\
  1 &\!\!\! 1 &\!\!\! 1 &  1 &\!\!\! 0 &\!\!\! 0 \\
  1 &\!\!\! 1 &\!\!\! 1 &  1 &\!\!\! 0 &\!\!\! 0
  \end{array} \right] } &
   36 &
  120
  \end{array}
  \]
  \smallskip
  \caption{The $3 \times 3 \times 2$ hyperdeterminant, orbits 81 to 120}
  \label{longtable3}
  \end{table}


  \begin{table}
  \[
  \begin{array}{rrrr|rrrr}
  \text{coef} & \text{representative}  & \text{size} & \#
  &
  \text{coef} & \text{representative}  & \text{size} & \#
  \\
   -20 &
  {\tiny \left[ \begin{array}{ccc|ccc}
  0 &\!\!\! 0 &\!\!\! 0 &  0 &\!\!\! 2 &\!\!\! 2 \\
  1 &\!\!\! 1 &\!\!\! 1 &  1 &\!\!\! 0 &\!\!\! 0 \\
  2 &\!\!\! 0 &\!\!\! 1 &  0 &\!\!\! 1 &\!\!\! 0
  \end{array} \right] } &
  144 &
  121
  &
     1 &
  {\tiny \left[ \begin{array}{ccc|ccc}
  0 &\!\!\! 0 &\!\!\! 0 &  0 &\!\!\! 2 &\!\!\! 2 \\
  2 &\!\!\! 0 &\!\!\! 0 &  0 &\!\!\! 0 &\!\!\! 2 \\
  2 &\!\!\! 2 &\!\!\! 0 &  0 &\!\!\! 0 &\!\!\! 0
  \end{array} \right] } &
   36 &
  122
  \\[8pt]
     4 &
  {\tiny \left[ \begin{array}{ccc|ccc}
  0 &\!\!\! 0 &\!\!\! 0 &  0 &\!\!\! 2 &\!\!\! 2 \\
  2 &\!\!\! 0 &\!\!\! 1 &  0 &\!\!\! 0 &\!\!\! 1 \\
  2 &\!\!\! 1 &\!\!\! 0 &  0 &\!\!\! 1 &\!\!\! 0
  \end{array} \right] } &
   36 &
  123
  &
     4 &
  {\tiny \left[ \begin{array}{ccc|ccc}
  0 &\!\!\! 0 &\!\!\! 0 &  0 &\!\!\! 2 &\!\!\! 2 \\
  2 &\!\!\! 0 &\!\!\! 0 &  0 &\!\!\! 1 &\!\!\! 1 \\
  2 &\!\!\! 1 &\!\!\! 1 &  0 &\!\!\! 0 &\!\!\! 0
  \end{array} \right] } &
   72 &
  124
  \\[8pt]
     4 &
  {\tiny \left[ \begin{array}{ccc|ccc}
  0 &\!\!\! 0 &\!\!\! 0 &  0 &\!\!\! 2 &\!\!\! 2 \\
  2 &\!\!\! 0 &\!\!\! 1 &  0 &\!\!\! 1 &\!\!\! 0 \\
  2 &\!\!\! 1 &\!\!\! 0 &  0 &\!\!\! 0 &\!\!\! 1
  \end{array} \right] } &
   36 &
  125
  &
     4 &
  {\tiny \left[ \begin{array}{ccc|ccc}
  0 &\!\!\! 0 &\!\!\! 1 &  0 &\!\!\! 2 &\!\!\! 1 \\
  0 &\!\!\! 0 &\!\!\! 2 &  2 &\!\!\! 0 &\!\!\! 0 \\
  1 &\!\!\! 2 &\!\!\! 0 &  1 &\!\!\! 0 &\!\!\! 0
  \end{array} \right] } &
   36 &
  126
  \\[8pt]
   -10 &
  {\tiny \left[ \begin{array}{ccc|ccc}
  0 &\!\!\! 0 &\!\!\! 1 &  0 &\!\!\! 2 &\!\!\! 1 \\
  0 &\!\!\! 0 &\!\!\! 2 &  2 &\!\!\! 0 &\!\!\! 0 \\
  2 &\!\!\! 1 &\!\!\! 0 &  0 &\!\!\! 1 &\!\!\! 0
  \end{array} \right] } &
   72 &
  127
  &
     8 &
  {\tiny \left[ \begin{array}{ccc|ccc}
  0 &\!\!\! 0 &\!\!\! 1 &  0 &\!\!\! 2 &\!\!\! 1 \\
  0 &\!\!\! 0 &\!\!\! 1 &  2 &\!\!\! 1 &\!\!\! 0 \\
  2 &\!\!\! 1 &\!\!\! 1 &  0 &\!\!\! 0 &\!\!\! 0
  \end{array} \right] } &
  144 &
  128
  \\[8pt]
   -20 &
  {\tiny \left[ \begin{array}{ccc|ccc}
  0 &\!\!\! 0 &\!\!\! 1 &  0 &\!\!\! 2 &\!\!\! 1 \\
  0 &\!\!\! 1 &\!\!\! 0 &  2 &\!\!\! 0 &\!\!\! 1 \\
  2 &\!\!\! 1 &\!\!\! 1 &  0 &\!\!\! 0 &\!\!\! 0
  \end{array} \right] } &
  144 &
  129
  &
   -12 &
  {\tiny \left[ \begin{array}{ccc|ccc}
  0 &\!\!\! 0 &\!\!\! 1 &  0 &\!\!\! 2 &\!\!\! 1 \\
  0 &\!\!\! 1 &\!\!\! 1 &  2 &\!\!\! 0 &\!\!\! 0 \\
  1 &\!\!\! 1 &\!\!\! 1 &  1 &\!\!\! 0 &\!\!\! 0
  \end{array} \right] } &
   72 &
  130
  \\[8pt]
    22 &
  {\tiny \left[ \begin{array}{ccc|ccc}
  0 &\!\!\! 0 &\!\!\! 1 &  0 &\!\!\! 2 &\!\!\! 1 \\
  0 &\!\!\! 1 &\!\!\! 1 &  2 &\!\!\! 0 &\!\!\! 0 \\
  2 &\!\!\! 0 &\!\!\! 1 &  0 &\!\!\! 1 &\!\!\! 0
  \end{array} \right] } &
  144 &
  131
  &
    -6 &
  {\tiny \left[ \begin{array}{ccc|ccc}
  0 &\!\!\! 0 &\!\!\! 1 &  0 &\!\!\! 2 &\!\!\! 1 \\
  0 &\!\!\! 1 &\!\!\! 1 &  2 &\!\!\! 0 &\!\!\! 0 \\
  2 &\!\!\! 1 &\!\!\! 0 &  0 &\!\!\! 0 &\!\!\! 1
  \end{array} \right] } &
  144 &
  132
  \\[8pt]
    18 &
  {\tiny \left[ \begin{array}{ccc|ccc}
  0 &\!\!\! 0 &\!\!\! 1 &  0 &\!\!\! 2 &\!\!\! 1 \\
  0 &\!\!\! 2 &\!\!\! 0 &  2 &\!\!\! 0 &\!\!\! 0 \\
  2 &\!\!\! 0 &\!\!\! 1 &  0 &\!\!\! 0 &\!\!\! 1
  \end{array} \right] } &
   72 &
  133
  &
   -26 &
  {\tiny \left[ \begin{array}{ccc|ccc}
  0 &\!\!\! 0 &\!\!\! 1 &  0 &\!\!\! 2 &\!\!\! 1 \\
  1 &\!\!\! 0 &\!\!\! 1 &  1 &\!\!\! 0 &\!\!\! 1 \\
  1 &\!\!\! 2 &\!\!\! 0 &  1 &\!\!\! 0 &\!\!\! 0
  \end{array} \right] } &
   72 &
  134
  \\[8pt]
    30 &
  {\tiny \left[ \begin{array}{ccc|ccc}
  0 &\!\!\! 0 &\!\!\! 1 &  0 &\!\!\! 2 &\!\!\! 1 \\
  1 &\!\!\! 0 &\!\!\! 1 &  1 &\!\!\! 0 &\!\!\! 1 \\
  2 &\!\!\! 1 &\!\!\! 0 &  0 &\!\!\! 1 &\!\!\! 0
  \end{array} \right] } &
   72 &
  135
  &
   -12 &
  {\tiny \left[ \begin{array}{ccc|ccc}
  0 &\!\!\! 0 &\!\!\! 1 &  0 &\!\!\! 2 &\!\!\! 1 \\
  1 &\!\!\! 0 &\!\!\! 0 &  1 &\!\!\! 1 &\!\!\! 1 \\
  2 &\!\!\! 1 &\!\!\! 1 &  0 &\!\!\! 0 &\!\!\! 0
  \end{array} \right] } &
  144 &
  136
  \\[8pt]
   -10 &
  {\tiny \left[ \begin{array}{ccc|ccc}
  0 &\!\!\! 0 &\!\!\! 1 &  0 &\!\!\! 2 &\!\!\! 1 \\
  1 &\!\!\! 0 &\!\!\! 1 &  1 &\!\!\! 1 &\!\!\! 0 \\
  1 &\!\!\! 1 &\!\!\! 1 &  1 &\!\!\! 0 &\!\!\! 0
  \end{array} \right] } &
  144 &
  137
  &
     2 &
  {\tiny \left[ \begin{array}{ccc|ccc}
  0 &\!\!\! 0 &\!\!\! 1 &  0 &\!\!\! 2 &\!\!\! 1 \\
  1 &\!\!\! 0 &\!\!\! 1 &  1 &\!\!\! 1 &\!\!\! 0 \\
  2 &\!\!\! 1 &\!\!\! 0 &  0 &\!\!\! 0 &\!\!\! 1
  \end{array} \right] } &
  144 &
  138
  \\[8pt]
    18 &
  {\tiny \left[ \begin{array}{ccc|ccc}
  0 &\!\!\! 0 &\!\!\! 1 &  0 &\!\!\! 2 &\!\!\! 1 \\
  1 &\!\!\! 1 &\!\!\! 0 &  1 &\!\!\! 0 &\!\!\! 1 \\
  1 &\!\!\! 1 &\!\!\! 1 &  1 &\!\!\! 0 &\!\!\! 0
  \end{array} \right] } &
  144 &
  139
  &
    16 &
  {\tiny \left[ \begin{array}{ccc|ccc}
  0 &\!\!\! 0 &\!\!\! 1 &  0 &\!\!\! 2 &\!\!\! 1 \\
  1 &\!\!\! 1 &\!\!\! 0 &  1 &\!\!\! 0 &\!\!\! 1 \\
  2 &\!\!\! 1 &\!\!\! 0 &  0 &\!\!\! 0 &\!\!\! 1
  \end{array} \right] } &
  144 &
  140
  \\[8pt]
   -26 &
  {\tiny \left[ \begin{array}{ccc|ccc}
  0 &\!\!\! 0 &\!\!\! 1 &  0 &\!\!\! 2 &\!\!\! 1 \\
  1 &\!\!\! 1 &\!\!\! 0 &  1 &\!\!\! 1 &\!\!\! 0 \\
  2 &\!\!\! 0 &\!\!\! 1 &  0 &\!\!\! 0 &\!\!\! 1
  \end{array} \right] } &
   36 &
  141
  &
    18 &
  {\tiny \left[ \begin{array}{ccc|ccc}
  0 &\!\!\! 0 &\!\!\! 1 &  0 &\!\!\! 2 &\!\!\! 1 \\
  1 &\!\!\! 2 &\!\!\! 0 &  1 &\!\!\! 0 &\!\!\! 0 \\
  2 &\!\!\! 0 &\!\!\! 0 &  0 &\!\!\! 0 &\!\!\! 2
  \end{array} \right] } &
   36 &
  142
  \\[8pt]
     4 &
  {\tiny \left[ \begin{array}{ccc|ccc}
  0 &\!\!\! 0 &\!\!\! 2 &  0 &\!\!\! 2 &\!\!\! 0 \\
  0 &\!\!\! 0 &\!\!\! 0 &  2 &\!\!\! 1 &\!\!\! 1 \\
  2 &\!\!\! 1 &\!\!\! 1 &  0 &\!\!\! 0 &\!\!\! 0
  \end{array} \right] } &
   72 &
  143
  &
   -24 &
  {\tiny \left[ \begin{array}{ccc|ccc}
  0 &\!\!\! 0 &\!\!\! 2 &  0 &\!\!\! 2 &\!\!\! 0 \\
  0 &\!\!\! 0 &\!\!\! 1 &  2 &\!\!\! 1 &\!\!\! 0 \\
  2 &\!\!\! 1 &\!\!\! 0 &  0 &\!\!\! 0 &\!\!\! 1
  \end{array} \right] } &
   72 &
  144
  \\[8pt]
    18 &
  {\tiny \left[ \begin{array}{ccc|ccc}
  0 &\!\!\! 0 &\!\!\! 2 &  0 &\!\!\! 2 &\!\!\! 0 \\
  0 &\!\!\! 1 &\!\!\! 0 &  2 &\!\!\! 0 &\!\!\! 1 \\
  2 &\!\!\! 1 &\!\!\! 0 &  0 &\!\!\! 0 &\!\!\! 1
  \end{array} \right] } &
   72 &
  145
  &
   -27 &
  {\tiny \left[ \begin{array}{ccc|ccc}
  0 &\!\!\! 0 &\!\!\! 2 &  0 &\!\!\! 2 &\!\!\! 0 \\
  0 &\!\!\! 2 &\!\!\! 0 &  2 &\!\!\! 0 &\!\!\! 0 \\
  2 &\!\!\! 0 &\!\!\! 0 &  0 &\!\!\! 0 &\!\!\! 2
  \end{array} \right] } &
   12 &
  146
  \\[8pt]
   -26 &
  {\tiny \left[ \begin{array}{ccc|ccc}
  0 &\!\!\! 0 &\!\!\! 2 &  0 &\!\!\! 2 &\!\!\! 0 \\
  1 &\!\!\! 0 &\!\!\! 0 &  1 &\!\!\! 1 &\!\!\! 1 \\
  1 &\!\!\! 1 &\!\!\! 1 &  1 &\!\!\! 0 &\!\!\! 0
  \end{array} \right] } &
   72 &
  147
  &
    16 &
  {\tiny \left[ \begin{array}{ccc|ccc}
  0 &\!\!\! 0 &\!\!\! 2 &  0 &\!\!\! 2 &\!\!\! 0 \\
  1 &\!\!\! 0 &\!\!\! 1 &  1 &\!\!\! 1 &\!\!\! 0 \\
  1 &\!\!\! 0 &\!\!\! 1 &  1 &\!\!\! 1 &\!\!\! 0
  \end{array} \right] } &
   36 &
  148
  \\[8pt]
    44 &
  {\tiny \left[ \begin{array}{ccc|ccc}
  0 &\!\!\! 0 &\!\!\! 2 &  0 &\!\!\! 2 &\!\!\! 0 \\
  1 &\!\!\! 0 &\!\!\! 1 &  1 &\!\!\! 1 &\!\!\! 0 \\
  1 &\!\!\! 1 &\!\!\! 0 &  1 &\!\!\! 0 &\!\!\! 1
  \end{array} \right] } &
   72 &
  149
  &
   -68 &
  {\tiny \left[ \begin{array}{ccc|ccc}
  0 &\!\!\! 0 &\!\!\! 2 &  0 &\!\!\! 2 &\!\!\! 0 \\
  1 &\!\!\! 1 &\!\!\! 0 &  1 &\!\!\! 0 &\!\!\! 1 \\
  1 &\!\!\! 1 &\!\!\! 0 &  1 &\!\!\! 0 &\!\!\! 1
  \end{array} \right] } &
   36 &
  150
  \\[8pt]
    36 &
  {\tiny \left[ \begin{array}{ccc|ccc}
  0 &\!\!\! 1 &\!\!\! 1 &  0 &\!\!\! 1 &\!\!\! 1 \\
  1 &\!\!\! 0 &\!\!\! 1 &  1 &\!\!\! 0 &\!\!\! 1 \\
  1 &\!\!\! 1 &\!\!\! 0 &  1 &\!\!\! 1 &\!\!\! 0
  \end{array} \right] } &
    6 &
  151
  &
    36 &
  {\tiny \left[ \begin{array}{ccc|ccc}
  0 &\!\!\! 1 &\!\!\! 1 &  0 &\!\!\! 1 &\!\!\! 1 \\
  1 &\!\!\! 0 &\!\!\! 0 &  1 &\!\!\! 1 &\!\!\! 1 \\
  1 &\!\!\! 1 &\!\!\! 1 &  1 &\!\!\! 0 &\!\!\! 0
  \end{array} \right] } &
   36 &
  152
  \\[8pt]
  -104 &
  {\tiny \left[ \begin{array}{ccc|ccc}
  0 &\!\!\! 1 &\!\!\! 1 &  0 &\!\!\! 1 &\!\!\! 1 \\
  1 &\!\!\! 0 &\!\!\! 1 &  1 &\!\!\! 1 &\!\!\! 0 \\
  1 &\!\!\! 1 &\!\!\! 0 &  1 &\!\!\! 0 &\!\!\! 1
  \end{array} \right] } &
   18 &
  153
  &
    -4 &
  {\tiny \left[ \begin{array}{ccc|ccc}
  0 &\!\!\! 0 &\!\!\! 0 &  1 &\!\!\! 1 &\!\!\! 2 \\
  0 &\!\!\! 1 &\!\!\! 1 &  1 &\!\!\! 1 &\!\!\! 0 \\
  2 &\!\!\! 1 &\!\!\! 1 &  0 &\!\!\! 0 &\!\!\! 0
  \end{array} \right] } &
   72 &
  154
  \\[8pt]
    -2 &
  {\tiny \left[ \begin{array}{ccc|ccc}
  0 &\!\!\! 0 &\!\!\! 0 &  1 &\!\!\! 1 &\!\!\! 2 \\
  0 &\!\!\! 2 &\!\!\! 0 &  1 &\!\!\! 0 &\!\!\! 1 \\
  2 &\!\!\! 1 &\!\!\! 1 &  0 &\!\!\! 0 &\!\!\! 0
  \end{array} \right] } &
  144 &
  155
  &
    -4 &
  {\tiny \left[ \begin{array}{ccc|ccc}
  0 &\!\!\! 0 &\!\!\! 0 &  1 &\!\!\! 1 &\!\!\! 2 \\
  0 &\!\!\! 2 &\!\!\! 1 &  1 &\!\!\! 0 &\!\!\! 0 \\
  1 &\!\!\! 1 &\!\!\! 1 &  1 &\!\!\! 0 &\!\!\! 0
  \end{array} \right] } &
  144 &
  156
  \\[8pt]
    26 &
  {\tiny \left[ \begin{array}{ccc|ccc}
  0 &\!\!\! 0 &\!\!\! 0 &  1 &\!\!\! 1 &\!\!\! 2 \\
  0 &\!\!\! 2 &\!\!\! 1 &  1 &\!\!\! 0 &\!\!\! 0 \\
  2 &\!\!\! 0 &\!\!\! 1 &  0 &\!\!\! 1 &\!\!\! 0
  \end{array} \right] } &
   72 &
  157
  &
   -16 &
  {\tiny \left[ \begin{array}{ccc|ccc}
  0 &\!\!\! 0 &\!\!\! 0 &  1 &\!\!\! 1 &\!\!\! 2 \\
  0 &\!\!\! 2 &\!\!\! 1 &  1 &\!\!\! 0 &\!\!\! 0 \\
  2 &\!\!\! 1 &\!\!\! 0 &  0 &\!\!\! 0 &\!\!\! 1
  \end{array} \right] } &
  144 &
  158
  \\[8pt]
    -4 &
  {\tiny \left[ \begin{array}{ccc|ccc}
  0 &\!\!\! 0 &\!\!\! 0 &  1 &\!\!\! 1 &\!\!\! 2 \\
  1 &\!\!\! 1 &\!\!\! 0 &  0 &\!\!\! 1 &\!\!\! 1 \\
  2 &\!\!\! 1 &\!\!\! 1 &  0 &\!\!\! 0 &\!\!\! 0
  \end{array} \right] } &
   72 &
  159
  &
    -8 &
  {\tiny \left[ \begin{array}{ccc|ccc}
  0 &\!\!\! 0 &\!\!\! 0 &  1 &\!\!\! 1 &\!\!\! 2 \\
  1 &\!\!\! 1 &\!\!\! 1 &  0 &\!\!\! 1 &\!\!\! 0 \\
  1 &\!\!\! 1 &\!\!\! 1 &  1 &\!\!\! 0 &\!\!\! 0
  \end{array} \right] } &
   72 &
  160
  \end{array}
  \]
  \smallskip
  \caption{The $3 \times 3 \times 2$ hyperdeterminant, orbits 121 to 160}
  \label{longtable4}
  \end{table}


  \begin{table}
  \[
  \begin{array}{rrrr|rrrr}
  \text{coef} & \text{representative}  & \text{size} & \#
  &
  \text{coef} & \text{representative}  & \text{size} & \#
  \\
    24 &
  {\tiny \left[ \begin{array}{ccc|ccc}
  0 &\!\!\! 0 &\!\!\! 0 &  1 &\!\!\! 1 &\!\!\! 2 \\
  1 &\!\!\! 1 &\!\!\! 1 &  0 &\!\!\! 1 &\!\!\! 0 \\
  2 &\!\!\! 1 &\!\!\! 0 &  0 &\!\!\! 0 &\!\!\! 1
  \end{array} \right] } &
  144 &
  161
  &
    -2 &
  {\tiny \left[ \begin{array}{ccc|ccc}
  0 &\!\!\! 0 &\!\!\! 0 &  1 &\!\!\! 1 &\!\!\! 2 \\
  1 &\!\!\! 2 &\!\!\! 0 &  0 &\!\!\! 0 &\!\!\! 1 \\
  2 &\!\!\! 1 &\!\!\! 0 &  0 &\!\!\! 0 &\!\!\! 1
  \end{array} \right] } &
   36 &
  162
  \\[8pt]
    54 &
  {\tiny \left[ \begin{array}{ccc|ccc}
  0 &\!\!\! 0 &\!\!\! 1 &  1 &\!\!\! 1 &\!\!\! 1 \\
  0 &\!\!\! 1 &\!\!\! 1 &  1 &\!\!\! 1 &\!\!\! 0 \\
  1 &\!\!\! 1 &\!\!\! 1 &  1 &\!\!\! 0 &\!\!\! 0
  \end{array} \right] } &
   36 &
  163
  &
   -22 &
  {\tiny \left[ \begin{array}{ccc|ccc}
  0 &\!\!\! 0 &\!\!\! 1 &  1 &\!\!\! 1 &\!\!\! 1 \\
  0 &\!\!\! 1 &\!\!\! 1 &  1 &\!\!\! 1 &\!\!\! 0 \\
  2 &\!\!\! 0 &\!\!\! 1 &  0 &\!\!\! 1 &\!\!\! 0
  \end{array} \right] } &
  144 &
  164
  \\[8pt]
    -8 &
  {\tiny \left[ \begin{array}{ccc|ccc}
  0 &\!\!\! 0 &\!\!\! 1 &  1 &\!\!\! 1 &\!\!\! 1 \\
  0 &\!\!\! 1 &\!\!\! 1 &  1 &\!\!\! 1 &\!\!\! 0 \\
  2 &\!\!\! 1 &\!\!\! 0 &  0 &\!\!\! 0 &\!\!\! 1
  \end{array} \right] } &
  144 &
  165
  &
    -4 &
  {\tiny \left[ \begin{array}{ccc|ccc}
  0 &\!\!\! 0 &\!\!\! 1 &  1 &\!\!\! 1 &\!\!\! 1 \\
  0 &\!\!\! 2 &\!\!\! 0 &  1 &\!\!\! 0 &\!\!\! 1 \\
  2 &\!\!\! 0 &\!\!\! 1 &  0 &\!\!\! 1 &\!\!\! 0
  \end{array} \right] } &
  144 &
  166
  \\[8pt]
    38 &
  {\tiny \left[ \begin{array}{ccc|ccc}
  0 &\!\!\! 0 &\!\!\! 1 &  1 &\!\!\! 1 &\!\!\! 1 \\
  0 &\!\!\! 2 &\!\!\! 0 &  1 &\!\!\! 0 &\!\!\! 1 \\
  2 &\!\!\! 1 &\!\!\! 0 &  0 &\!\!\! 0 &\!\!\! 1
  \end{array} \right] } &
   72 &
  167
  &
     6 &
  {\tiny \left[ \begin{array}{ccc|ccc}
  0 &\!\!\! 0 &\!\!\! 1 &  1 &\!\!\! 1 &\!\!\! 1 \\
  0 &\!\!\! 2 &\!\!\! 1 &  1 &\!\!\! 0 &\!\!\! 0 \\
  1 &\!\!\! 1 &\!\!\! 0 &  1 &\!\!\! 0 &\!\!\! 1
  \end{array} \right] } &
   72 &
  168
  \\[8pt]
    10 &
  {\tiny \left[ \begin{array}{ccc|ccc}
  0 &\!\!\! 0 &\!\!\! 1 &  1 &\!\!\! 1 &\!\!\! 1 \\
  0 &\!\!\! 1 &\!\!\! 1 &  2 &\!\!\! 0 &\!\!\! 0 \\
  1 &\!\!\! 2 &\!\!\! 0 &  0 &\!\!\! 0 &\!\!\! 1
  \end{array} \right] } &
  144 &
  169
  &
    -2 &
  {\tiny \left[ \begin{array}{ccc|ccc}
  0 &\!\!\! 0 &\!\!\! 1 &  1 &\!\!\! 1 &\!\!\! 1 \\
  1 &\!\!\! 1 &\!\!\! 0 &  0 &\!\!\! 1 &\!\!\! 1 \\
  1 &\!\!\! 1 &\!\!\! 1 &  1 &\!\!\! 0 &\!\!\! 0
  \end{array} \right] } &
   72 &
  170
  \\[8pt]
   -36 &
  {\tiny \left[ \begin{array}{ccc|ccc}
  0 &\!\!\! 0 &\!\!\! 1 &  1 &\!\!\! 1 &\!\!\! 1 \\
  1 &\!\!\! 1 &\!\!\! 0 &  0 &\!\!\! 1 &\!\!\! 1 \\
  2 &\!\!\! 1 &\!\!\! 0 &  0 &\!\!\! 0 &\!\!\! 1
  \end{array} \right] } &
   72 &
  171
  &
    12 &
  {\tiny \left[ \begin{array}{ccc|ccc}
  0 &\!\!\! 0 &\!\!\! 2 &  1 &\!\!\! 1 &\!\!\! 0 \\
  0 &\!\!\! 0 &\!\!\! 1 &  1 &\!\!\! 2 &\!\!\! 0 \\
  2 &\!\!\! 1 &\!\!\! 0 &  0 &\!\!\! 0 &\!\!\! 1
  \end{array} \right] } &
   72 &
  172
  \\[8pt]
   -30 &
  {\tiny \left[ \begin{array}{ccc|ccc}
  0 &\!\!\! 0 &\!\!\! 2 &  1 &\!\!\! 1 &\!\!\! 0 \\
  0 &\!\!\! 2 &\!\!\! 0 &  1 &\!\!\! 0 &\!\!\! 1 \\
  2 &\!\!\! 0 &\!\!\! 0 &  0 &\!\!\! 1 &\!\!\! 1
  \end{array} \right] } &
   12 &
  173
  &
   -16 &
  {\tiny \left[ \begin{array}{ccc|ccc}
  0 &\!\!\! 0 &\!\!\! 2 &  1 &\!\!\! 1 &\!\!\! 0 \\
  0 &\!\!\! 1 &\!\!\! 0 &  2 &\!\!\! 0 &\!\!\! 1 \\
  1 &\!\!\! 2 &\!\!\! 0 &  0 &\!\!\! 0 &\!\!\! 1
  \end{array} \right] } &
   72 &
  174
  \\[8pt]
   -32 &
  {\tiny \left[ \begin{array}{ccc|ccc}
  0 &\!\!\! 0 &\!\!\! 2 &  1 &\!\!\! 1 &\!\!\! 0 \\
  0 &\!\!\! 1 &\!\!\! 1 &  2 &\!\!\! 0 &\!\!\! 0 \\
  1 &\!\!\! 1 &\!\!\! 0 &  0 &\!\!\! 1 &\!\!\! 1
  \end{array} \right] } &
   72 &
  175
  &
    76 &
  {\tiny \left[ \begin{array}{ccc|ccc}
  0 &\!\!\! 0 &\!\!\! 2 &  1 &\!\!\! 1 &\!\!\! 0 \\
  1 &\!\!\! 1 &\!\!\! 0 &  0 &\!\!\! 1 &\!\!\! 1 \\
  1 &\!\!\! 1 &\!\!\! 0 &  1 &\!\!\! 0 &\!\!\! 1
  \end{array} \right] } &
   36 &
  176
  \\[8pt]
    10 &
  {\tiny \left[ \begin{array}{ccc|ccc}
  0 &\!\!\! 0 &\!\!\! 1 &  1 &\!\!\! 2 &\!\!\! 0 \\
  1 &\!\!\! 0 &\!\!\! 1 &  0 &\!\!\! 1 &\!\!\! 1 \\
  2 &\!\!\! 1 &\!\!\! 0 &  0 &\!\!\! 0 &\!\!\! 1
  \end{array} \right] } &
   72 &
  177
  &
   -30 &
  {\tiny \left[ \begin{array}{ccc|ccc}
  0 &\!\!\! 1 &\!\!\! 1 &  1 &\!\!\! 0 &\!\!\! 1 \\
  1 &\!\!\! 1 &\!\!\! 0 &  0 &\!\!\! 1 &\!\!\! 1 \\
  1 &\!\!\! 0 &\!\!\! 1 &  1 &\!\!\! 1 &\!\!\! 0
  \end{array} \right] } &
   12 &
  178
  \end{array}
  \]
  \smallskip
  \caption{The $3 \times 3 \times 2$ hyperdeterminant, orbits 161 to 178}
  \label{longtable5}
  \end{table}


\begin{table}
\begin{verbatim}
result := []:
for o to nops(invariant) do
  c := invariant[o][1]:
  for i in orbitlist[o] do
    x := unflatten( weightzero[i] ):
    for j to 3 do for k to 2 do
      if x[2,j,k] >= 1 then
        xx := copy(x):
        xx[2,j,k] := xx[2,j,k] - 1:
        xx[1,j,k] := xx[1,j,k] + 1:
        result := [ op(result), [ c*x[2,j,k], flatten(xx) ] ]
      fi
    od od
  od:
  result := compress( result )
od:
\end{verbatim}
\caption{Maple code to verify results with integer arithmetic}
\label{mapleinteger}
\end{table}


  \begin{table}
  \[
  \begin{array}{rrrr|rrrr}
  & \text{representative}  & \text{size} & \#
  &
  & \text{representative}  & \text{size} & \#
  \\
     1 &
  {\tiny \left[ \begin{array}{cccc|cccc}
  0 &\!\!\! 0 &\!\!\! 0 &\!\!\! 0 & 0 &\!\!\! 0 &\!\!\! 0 &\!\!\! 2 \\
  0 &\!\!\! 0 &\!\!\! 0 &\!\!\! 0 & 0 &\!\!\! 0 &\!\!\! 2 &\!\!\! 0 \\
  0 &\!\!\! 2 &\!\!\! 0 &\!\!\! 0 & 0 &\!\!\! 0 &\!\!\! 0 &\!\!\! 0 \\
  2 &\!\!\! 0 &\!\!\! 0 &\!\!\! 0 & 0 &\!\!\! 0 &\!\!\! 0 &\!\!\! 0
  \end{array} \right] } &
   144 &
     1
  &
    -2 &
  {\tiny \left[ \begin{array}{cccc|cccc}
  0 &\!\!\! 0 &\!\!\! 0 &\!\!\! 0 & 0 &\!\!\! 0 &\!\!\! 0 &\!\!\! 2 \\
  0 &\!\!\! 0 &\!\!\! 0 &\!\!\! 0 & 0 &\!\!\! 0 &\!\!\! 2 &\!\!\! 0 \\
  1 &\!\!\! 1 &\!\!\! 0 &\!\!\! 0 & 0 &\!\!\! 0 &\!\!\! 0 &\!\!\! 0 \\
  1 &\!\!\! 1 &\!\!\! 0 &\!\!\! 0 & 0 &\!\!\! 0 &\!\!\! 0 &\!\!\! 0
  \end{array} \right] } &
   144 &
     2
  \\[12pt]
    -1 &
  {\tiny \left[ \begin{array}{cccc|cccc}
  0 &\!\!\! 0 &\!\!\! 0 &\!\!\! 0 & 0 &\!\!\! 0 &\!\!\! 0 &\!\!\! 2 \\
  0 &\!\!\! 0 &\!\!\! 1 &\!\!\! 0 & 0 &\!\!\! 0 &\!\!\! 1 &\!\!\! 0 \\
  0 &\!\!\! 1 &\!\!\! 0 &\!\!\! 0 & 0 &\!\!\! 1 &\!\!\! 0 &\!\!\! 0 \\
  2 &\!\!\! 0 &\!\!\! 0 &\!\!\! 0 & 0 &\!\!\! 0 &\!\!\! 0 &\!\!\! 0
  \end{array} \right] } &
   288 &
     3
  &
     1 &
  {\tiny \left[ \begin{array}{cccc|cccc}
  0 &\!\!\! 0 &\!\!\! 0 &\!\!\! 0 & 0 &\!\!\! 0 &\!\!\! 0 &\!\!\! 2 \\
  0 &\!\!\! 0 &\!\!\! 1 &\!\!\! 0 & 0 &\!\!\! 0 &\!\!\! 1 &\!\!\! 0 \\
  0 &\!\!\! 1 &\!\!\! 0 &\!\!\! 0 & 1 &\!\!\! 0 &\!\!\! 0 &\!\!\! 0 \\
  1 &\!\!\! 1 &\!\!\! 0 &\!\!\! 0 & 0 &\!\!\! 0 &\!\!\! 0 &\!\!\! 0
  \end{array} \right] } &
  1152 &
     4
  \\[12pt]
    -2 &
  {\tiny \left[ \begin{array}{cccc|cccc}
  0 &\!\!\! 0 &\!\!\! 0 &\!\!\! 0 & 0 &\!\!\! 0 &\!\!\! 0 &\!\!\! 2 \\
  0 &\!\!\! 0 &\!\!\! 2 &\!\!\! 0 & 0 &\!\!\! 0 &\!\!\! 0 &\!\!\! 0 \\
  0 &\!\!\! 0 &\!\!\! 0 &\!\!\! 0 & 1 &\!\!\! 1 &\!\!\! 0 &\!\!\! 0 \\
  1 &\!\!\! 1 &\!\!\! 0 &\!\!\! 0 & 0 &\!\!\! 0 &\!\!\! 0 &\!\!\! 0
  \end{array} \right] } &
   576 &
     5
  &
     3 &
  {\tiny \left[ \begin{array}{cccc|cccc}
  0 &\!\!\! 0 &\!\!\! 0 &\!\!\! 0 & 0 &\!\!\! 0 &\!\!\! 0 &\!\!\! 2 \\
  0 &\!\!\! 0 &\!\!\! 2 &\!\!\! 0 & 0 &\!\!\! 0 &\!\!\! 0 &\!\!\! 0 \\
  0 &\!\!\! 1 &\!\!\! 0 &\!\!\! 0 & 1 &\!\!\! 0 &\!\!\! 0 &\!\!\! 0 \\
  1 &\!\!\! 0 &\!\!\! 0 &\!\!\! 0 & 0 &\!\!\! 1 &\!\!\! 0 &\!\!\! 0
  \end{array} \right] } &
   288 &
     6
  \\[12pt]
     2 &
  {\tiny \left[ \begin{array}{cccc|cccc}
  0 &\!\!\! 0 &\!\!\! 0 &\!\!\! 0 & 0 &\!\!\! 0 &\!\!\! 0 &\!\!\! 2 \\
  0 &\!\!\! 0 &\!\!\! 0 &\!\!\! 0 & 0 &\!\!\! 1 &\!\!\! 1 &\!\!\! 0 \\
  1 &\!\!\! 0 &\!\!\! 1 &\!\!\! 0 & 0 &\!\!\! 0 &\!\!\! 0 &\!\!\! 0 \\
  1 &\!\!\! 1 &\!\!\! 0 &\!\!\! 0 & 0 &\!\!\! 0 &\!\!\! 0 &\!\!\! 0
  \end{array} \right] } &
  1152 &
     7
  &
     2 &
  {\tiny \left[ \begin{array}{cccc|cccc}
  0 &\!\!\! 0 &\!\!\! 0 &\!\!\! 0 & 0 &\!\!\! 0 &\!\!\! 0 &\!\!\! 2 \\
  0 &\!\!\! 0 &\!\!\! 1 &\!\!\! 0 & 0 &\!\!\! 1 &\!\!\! 0 &\!\!\! 0 \\
  0 &\!\!\! 0 &\!\!\! 1 &\!\!\! 0 & 1 &\!\!\! 0 &\!\!\! 0 &\!\!\! 0 \\
  1 &\!\!\! 1 &\!\!\! 0 &\!\!\! 0 & 0 &\!\!\! 0 &\!\!\! 0 &\!\!\! 0
  \end{array} \right] } &
   576 &
     8
  \\[12pt]
    -3 &
  {\tiny \left[ \begin{array}{cccc|cccc}
  0 &\!\!\! 0 &\!\!\! 0 &\!\!\! 0 & 0 &\!\!\! 0 &\!\!\! 0 &\!\!\! 2 \\
  0 &\!\!\! 0 &\!\!\! 1 &\!\!\! 0 & 0 &\!\!\! 1 &\!\!\! 0 &\!\!\! 0 \\
  0 &\!\!\! 1 &\!\!\! 0 &\!\!\! 0 & 1 &\!\!\! 0 &\!\!\! 0 &\!\!\! 0 \\
  1 &\!\!\! 0 &\!\!\! 1 &\!\!\! 0 & 0 &\!\!\! 0 &\!\!\! 0 &\!\!\! 0
  \end{array} \right] } &
  1152 &
     9
  &
     6 &
  {\tiny \left[ \begin{array}{cccc|cccc}
  0 &\!\!\! 0 &\!\!\! 0 &\!\!\! 1 & 0 &\!\!\! 0 &\!\!\! 0 &\!\!\! 1 \\
  0 &\!\!\! 0 &\!\!\! 1 &\!\!\! 0 & 0 &\!\!\! 0 &\!\!\! 1 &\!\!\! 0 \\
  0 &\!\!\! 1 &\!\!\! 0 &\!\!\! 0 & 0 &\!\!\! 1 &\!\!\! 0 &\!\!\! 0 \\
  1 &\!\!\! 0 &\!\!\! 0 &\!\!\! 0 & 1 &\!\!\! 0 &\!\!\! 0 &\!\!\! 0
  \end{array} \right] } &
    24 &
    10
  \\[12pt]
     2 &
  {\tiny \left[ \begin{array}{cccc|cccc}
  0 &\!\!\! 0 &\!\!\! 0 &\!\!\! 1 & 0 &\!\!\! 0 &\!\!\! 0 &\!\!\! 1 \\
  0 &\!\!\! 0 &\!\!\! 1 &\!\!\! 0 & 0 &\!\!\! 0 &\!\!\! 1 &\!\!\! 0 \\
  0 &\!\!\! 0 &\!\!\! 0 &\!\!\! 0 & 1 &\!\!\! 1 &\!\!\! 0 &\!\!\! 0 \\
  1 &\!\!\! 1 &\!\!\! 0 &\!\!\! 0 & 0 &\!\!\! 0 &\!\!\! 0 &\!\!\! 0
  \end{array} \right] } &
   288 &
    11
  &
    -8 &
  {\tiny \left[ \begin{array}{cccc|cccc}
  0 &\!\!\! 0 &\!\!\! 0 &\!\!\! 1 & 0 &\!\!\! 0 &\!\!\! 0 &\!\!\! 1 \\
  0 &\!\!\! 0 &\!\!\! 1 &\!\!\! 0 & 0 &\!\!\! 0 &\!\!\! 1 &\!\!\! 0 \\
  0 &\!\!\! 1 &\!\!\! 0 &\!\!\! 0 & 1 &\!\!\! 0 &\!\!\! 0 &\!\!\! 0 \\
  1 &\!\!\! 0 &\!\!\! 0 &\!\!\! 0 & 0 &\!\!\! 1 &\!\!\! 0 &\!\!\! 0
  \end{array} \right] } &
   144 &
    12
  \\[12pt]
    -1 &
  {\tiny \left[ \begin{array}{cccc|cccc}
  0 &\!\!\! 0 &\!\!\! 0 &\!\!\! 1 & 0 &\!\!\! 0 &\!\!\! 0 &\!\!\! 1 \\
  0 &\!\!\! 0 &\!\!\! 0 &\!\!\! 0 & 0 &\!\!\! 1 &\!\!\! 1 &\!\!\! 0 \\
  0 &\!\!\! 0 &\!\!\! 1 &\!\!\! 0 & 1 &\!\!\! 0 &\!\!\! 0 &\!\!\! 0 \\
  1 &\!\!\! 1 &\!\!\! 0 &\!\!\! 0 & 0 &\!\!\! 0 &\!\!\! 0 &\!\!\! 0
  \end{array} \right] } &
  1152 &
    13
  &
    -1 &
  {\tiny \left[ \begin{array}{cccc|cccc}
  0 &\!\!\! 0 &\!\!\! 0 &\!\!\! 1 & 0 &\!\!\! 0 &\!\!\! 0 &\!\!\! 1 \\
  0 &\!\!\! 0 &\!\!\! 0 &\!\!\! 0 & 0 &\!\!\! 1 &\!\!\! 1 &\!\!\! 0 \\
  1 &\!\!\! 0 &\!\!\! 0 &\!\!\! 0 & 0 &\!\!\! 0 &\!\!\! 1 &\!\!\! 0 \\
  1 &\!\!\! 1 &\!\!\! 0 &\!\!\! 0 & 0 &\!\!\! 0 &\!\!\! 0 &\!\!\! 0
  \end{array} \right] } &
   576 &
    14
  \\[12pt]
     9 &
  {\tiny \left[ \begin{array}{cccc|cccc}
  0 &\!\!\! 0 &\!\!\! 0 &\!\!\! 1 & 0 &\!\!\! 0 &\!\!\! 0 &\!\!\! 1 \\
  0 &\!\!\! 0 &\!\!\! 1 &\!\!\! 0 & 0 &\!\!\! 1 &\!\!\! 0 &\!\!\! 0 \\
  0 &\!\!\! 1 &\!\!\! 0 &\!\!\! 0 & 1 &\!\!\! 0 &\!\!\! 0 &\!\!\! 0 \\
  1 &\!\!\! 0 &\!\!\! 0 &\!\!\! 0 & 0 &\!\!\! 0 &\!\!\! 1 &\!\!\! 0
  \end{array} \right] } &
   192 &
    15
  &
     4 &
  {\tiny \left[ \begin{array}{cccc|cccc}
  0 &\!\!\! 0 &\!\!\! 0 &\!\!\! 0 & 0 &\!\!\! 0 &\!\!\! 1 &\!\!\! 1 \\
  0 &\!\!\! 0 &\!\!\! 0 &\!\!\! 0 & 0 &\!\!\! 0 &\!\!\! 1 &\!\!\! 1 \\
  1 &\!\!\! 1 &\!\!\! 0 &\!\!\! 0 & 0 &\!\!\! 0 &\!\!\! 0 &\!\!\! 0 \\
  1 &\!\!\! 1 &\!\!\! 0 &\!\!\! 0 & 0 &\!\!\! 0 &\!\!\! 0 &\!\!\! 0
  \end{array} \right] } &
    36 &
    16
  \\[12pt]
    -1 &
  {\tiny \left[ \begin{array}{cccc|cccc}
  0 &\!\!\! 0 &\!\!\! 0 &\!\!\! 0 & 0 &\!\!\! 0 &\!\!\! 1 &\!\!\! 1 \\
  0 &\!\!\! 0 &\!\!\! 0 &\!\!\! 1 & 0 &\!\!\! 0 &\!\!\! 1 &\!\!\! 0 \\
  0 &\!\!\! 1 &\!\!\! 0 &\!\!\! 0 & 1 &\!\!\! 0 &\!\!\! 0 &\!\!\! 0 \\
  1 &\!\!\! 1 &\!\!\! 0 &\!\!\! 0 & 0 &\!\!\! 0 &\!\!\! 0 &\!\!\! 0
  \end{array} \right] } &
   576 &
    17
  &
     4 &
  {\tiny \left[ \begin{array}{cccc|cccc}
  0 &\!\!\! 0 &\!\!\! 0 &\!\!\! 0 & 0 &\!\!\! 0 &\!\!\! 1 &\!\!\! 1 \\
  0 &\!\!\! 0 &\!\!\! 1 &\!\!\! 1 & 0 &\!\!\! 0 &\!\!\! 0 &\!\!\! 0 \\
  0 &\!\!\! 0 &\!\!\! 0 &\!\!\! 0 & 1 &\!\!\! 1 &\!\!\! 0 &\!\!\! 0 \\
  1 &\!\!\! 1 &\!\!\! 0 &\!\!\! 0 & 0 &\!\!\! 0 &\!\!\! 0 &\!\!\! 0
  \end{array} \right] } &
   144 &
    18
  \\[12pt]
     4 &
  {\tiny \left[ \begin{array}{cccc|cccc}
  0 &\!\!\! 0 &\!\!\! 0 &\!\!\! 0 & 0 &\!\!\! 0 &\!\!\! 1 &\!\!\! 1 \\
  0 &\!\!\! 0 &\!\!\! 1 &\!\!\! 1 & 0 &\!\!\! 0 &\!\!\! 0 &\!\!\! 0 \\
  0 &\!\!\! 1 &\!\!\! 0 &\!\!\! 0 & 1 &\!\!\! 0 &\!\!\! 0 &\!\!\! 0 \\
  0 &\!\!\! 1 &\!\!\! 0 &\!\!\! 0 & 1 &\!\!\! 0 &\!\!\! 0 &\!\!\! 0
  \end{array} \right] } &
   144 &
    19
  &
    -6 &
  {\tiny \left[ \begin{array}{cccc|cccc}
  0 &\!\!\! 0 &\!\!\! 0 &\!\!\! 0 & 0 &\!\!\! 0 &\!\!\! 1 &\!\!\! 1 \\
  0 &\!\!\! 0 &\!\!\! 1 &\!\!\! 1 & 0 &\!\!\! 0 &\!\!\! 0 &\!\!\! 0 \\
  0 &\!\!\! 1 &\!\!\! 0 &\!\!\! 0 & 1 &\!\!\! 0 &\!\!\! 0 &\!\!\! 0 \\
  1 &\!\!\! 0 &\!\!\! 0 &\!\!\! 0 & 0 &\!\!\! 1 &\!\!\! 0 &\!\!\! 0
  \end{array} \right] } &
   288 &
    20
  \\[12pt]
    -2 &
  {\tiny \left[ \begin{array}{cccc|cccc}
  0 &\!\!\! 0 &\!\!\! 0 &\!\!\! 0 & 0 &\!\!\! 0 &\!\!\! 1 &\!\!\! 1 \\
  0 &\!\!\! 0 &\!\!\! 0 &\!\!\! 0 & 0 &\!\!\! 1 &\!\!\! 0 &\!\!\! 1 \\
  1 &\!\!\! 0 &\!\!\! 1 &\!\!\! 0 & 0 &\!\!\! 0 &\!\!\! 0 &\!\!\! 0 \\
  1 &\!\!\! 1 &\!\!\! 0 &\!\!\! 0 & 0 &\!\!\! 0 &\!\!\! 0 &\!\!\! 0
  \end{array} \right] } &
   576 &
    21
  &
    -2 &
  {\tiny \left[ \begin{array}{cccc|cccc}
  0 &\!\!\! 0 &\!\!\! 0 &\!\!\! 0 & 0 &\!\!\! 0 &\!\!\! 1 &\!\!\! 1 \\
  0 &\!\!\! 0 &\!\!\! 0 &\!\!\! 1 & 0 &\!\!\! 1 &\!\!\! 0 &\!\!\! 0 \\
  0 &\!\!\! 0 &\!\!\! 1 &\!\!\! 0 & 1 &\!\!\! 0 &\!\!\! 0 &\!\!\! 0 \\
  1 &\!\!\! 1 &\!\!\! 0 &\!\!\! 0 & 0 &\!\!\! 0 &\!\!\! 0 &\!\!\! 0
  \end{array} \right] } &
   576 &
    22
  \\[12pt]
     3 &
  {\tiny \left[ \begin{array}{cccc|cccc}
  0 &\!\!\! 0 &\!\!\! 0 &\!\!\! 0 & 0 &\!\!\! 0 &\!\!\! 1 &\!\!\! 1 \\
  0 &\!\!\! 0 &\!\!\! 0 &\!\!\! 1 & 0 &\!\!\! 1 &\!\!\! 0 &\!\!\! 0 \\
  0 &\!\!\! 1 &\!\!\! 0 &\!\!\! 0 & 1 &\!\!\! 0 &\!\!\! 0 &\!\!\! 0 \\
  1 &\!\!\! 0 &\!\!\! 1 &\!\!\! 0 & 0 &\!\!\! 0 &\!\!\! 0 &\!\!\! 0
  \end{array} \right] } &
  1152 &
    23
  &
     3 &
  {\tiny \left[ \begin{array}{cccc|cccc}
  0 &\!\!\! 0 &\!\!\! 0 &\!\!\! 0 & 0 &\!\!\! 0 &\!\!\! 1 &\!\!\! 1 \\
  0 &\!\!\! 0 &\!\!\! 0 &\!\!\! 1 & 0 &\!\!\! 1 &\!\!\! 0 &\!\!\! 0 \\
  1 &\!\!\! 0 &\!\!\! 0 &\!\!\! 0 & 0 &\!\!\! 0 &\!\!\! 1 &\!\!\! 0 \\
  1 &\!\!\! 1 &\!\!\! 0 &\!\!\! 0 & 0 &\!\!\! 0 &\!\!\! 0 &\!\!\! 0
  \end{array} \right] } &
  1152 &
    24
  \\[12pt]
    -2 &
  {\tiny \left[ \begin{array}{cccc|cccc}
  0 &\!\!\! 0 &\!\!\! 0 &\!\!\! 0 & 0 &\!\!\! 0 &\!\!\! 1 &\!\!\! 1 \\
  0 &\!\!\! 0 &\!\!\! 0 &\!\!\! 1 & 0 &\!\!\! 1 &\!\!\! 0 &\!\!\! 0 \\
  1 &\!\!\! 0 &\!\!\! 1 &\!\!\! 0 & 0 &\!\!\! 0 &\!\!\! 0 &\!\!\! 0 \\
  1 &\!\!\! 0 &\!\!\! 0 &\!\!\! 0 & 0 &\!\!\! 1 &\!\!\! 0 &\!\!\! 0
  \end{array} \right] } &
  1152 &
    25
  &
    -2 &
  {\tiny \left[ \begin{array}{cccc|cccc}
  0 &\!\!\! 0 &\!\!\! 0 &\!\!\! 0 & 0 &\!\!\! 0 &\!\!\! 1 &\!\!\! 1 \\
  0 &\!\!\! 1 &\!\!\! 0 &\!\!\! 1 & 0 &\!\!\! 0 &\!\!\! 0 &\!\!\! 0 \\
  0 &\!\!\! 0 &\!\!\! 0 &\!\!\! 0 & 1 &\!\!\! 1 &\!\!\! 0 &\!\!\! 0 \\
  1 &\!\!\! 0 &\!\!\! 1 &\!\!\! 0 & 0 &\!\!\! 0 &\!\!\! 0 &\!\!\! 0
  \end{array} \right] } &
   288 &
    26
  \\[12pt]
    14 &
  {\tiny \left[ \begin{array}{cccc|cccc}
  0 &\!\!\! 0 &\!\!\! 0 &\!\!\! 1 & 0 &\!\!\! 0 &\!\!\! 1 &\!\!\! 0 \\
  0 &\!\!\! 0 &\!\!\! 1 &\!\!\! 0 & 0 &\!\!\! 0 &\!\!\! 0 &\!\!\! 1 \\
  0 &\!\!\! 1 &\!\!\! 0 &\!\!\! 0 & 1 &\!\!\! 0 &\!\!\! 0 &\!\!\! 0 \\
  1 &\!\!\! 0 &\!\!\! 0 &\!\!\! 0 & 0 &\!\!\! 1 &\!\!\! 0 &\!\!\! 0
  \end{array} \right] } &
    72 &
    27
  &
   -12 &
  {\tiny \left[ \begin{array}{cccc|cccc}
  0 &\!\!\! 0 &\!\!\! 0 &\!\!\! 1 & 0 &\!\!\! 0 &\!\!\! 1 &\!\!\! 0 \\
  0 &\!\!\! 0 &\!\!\! 1 &\!\!\! 0 & 0 &\!\!\! 1 &\!\!\! 0 &\!\!\! 0 \\
  0 &\!\!\! 1 &\!\!\! 0 &\!\!\! 0 & 1 &\!\!\! 0 &\!\!\! 0 &\!\!\! 0 \\
  1 &\!\!\! 0 &\!\!\! 0 &\!\!\! 0 & 0 &\!\!\! 0 &\!\!\! 0 &\!\!\! 1
  \end{array} \right] } &
   144 &
    28
  \end{array}
  \]
  \smallskip
  \caption{The simplest invariant of a $4 \times 4 \times 2$ array}
  \label{table442}
  \end{table}


\section*{Acknowledgement}

The author was supported by a Discovery Grant from NSERC, the Natural Sciences and Engineering Research Council of Canada.


\end{document}